\documentclass{amsart}
\usepackage{amssymb}
\usepackage{amsmath}
\usepackage{amsfonts}

\newtheorem{theorem}{Theorem}
\theoremstyle{plain}

\newtheorem{corollary}{Corollary}

\newtheorem{definition}{Definition}

\newtheorem{lemma}{Lemma}
\newtheorem{notation}{Notation}

\newtheorem{proposition}{Proposition}
\newtheorem{remark}{Remark}

\numberwithin{equation}{section}
\input{tcilatex}

\begin{document}
\title[On uniqueness of solutions of Navier-Stokes equations]{On uniqueness
of weak solutions of the incompressible Navier-Stokes equations}
\author{Kamal N. Soltanov}
\address{{\small National Academy of Sciences of Azerbaijan, Baku,
AZERBAIJAN }}
\email{sultan\_kamal@hotmail.com}
\urladdr{}
\date{}

\begin{abstract}
In this article the question on uniqueness of weak solution of the
incompressible Navier-Stokes Equations in the 3-dimensional case is studied.
Here the investigation is carried out with use of another approach. The
uniqueness of velocity for the considered problem is proved for given
functions from spaces that possesess some smoothness. Moreover, these spaces
are dense in respective spaces of functions, under which were proved
existence of the weak solutions. In addition here the solvability and
uniqueness of the weak solutions of auxiliary problems associated with the
main problem is investigated, and also one conditional result on uniqueness
is proved.
\end{abstract}

\maketitle

\section{\label{Sec_1}Introduction}

In this article we investigate the question on the uniqueness of the weak
solutions of the incompressible Navier-Stokes equations, namely is
investigated question: when the weak solution of the following problem is
unique? 
\begin{equation}
\frac{\partial u_{i}}{\partial t}-\nu \Delta u_{i}+\underset{j=1}{\overset{n}{\sum }}u_{j}\frac{\partial u_{i}}{\partial x_{j}}+\frac{\partial p}{\partial x_{i}}=f_{i},\quad i=\overline{1,n},
\label{1}
\end{equation}%
\begin{equation}
\func{div}u=\underset{i=1}{\overset{n}{\sum }}\frac{\partial u_{i}}{\partial x_{i}}=0,\quad x\in \Omega \subset R^{n},t>0\quad ,
\label{2}
\end{equation}%
\begin{equation}
u\left( 0,x\right) =u_{0}\left( x\right) ,\quad x\in \Omega ;\quad u\left\vert \ _{\left( 0,T\right) \times \partial \Omega }\right. =0
\label{3}
\end{equation}%
where $\Omega \subset R^{n}$ is a bounded domain with sufficiently smooth
boundary $\partial \Omega $, $T>0$ is a positive number. In this work for
study of the posed question two distinct way are used, therefore it consist
of two parts.

As is well-known Navier-Stokes equations describe the motion of a fluid in $%
R^{n}$ ($n=2$ or $3$). Consequently, in this problem $u(x,t)=\left\{
u_{i}(x,t)\right\} _{1}^{n}\in R^{n}$ is an unknown velocity vector and $%
p(x,t)\in R$ is an unknown pressure, at the position $x\in R^{n}$ and time $%
t\geq 0$; $f_{i}(x,t)$ are the components of a given, externally applied
force (e.g. gravity), $\nu $ is a positive coefficient (the viscosity), $%
u_{0}\left( x\right) \in R^{n}$ is a sufficiently smooth vector function
(vector field).

As is well-known in \cite{Ler1} is shown that the Navier-Stokes equations (%
\ref{1}), (\ref{2}), (\ref{3}) in three dimensions case has a weak solution $%
(u,p)$ with suitable properties (see, also, \cite{Hop}, \cite{Lad1}, \cite%
{MajBer}, \cite{Con1}, \cite{Fef1}, \cite{Lio1}). It is known that
uniqueness of weak solution of the Navier-Stokes equation in two space
dimensions case were proved (\cite{LioPro}, \cite{Lio1}, see also \cite{Lad2}%
), but the result of such type for the uniqueness of weak solutions in three
space dimensions case as yet isn't known. It should be noted that in three
dimensional case the uniqueness was studied also, but under complementary
conditions on the smoothness of the solution (see, e.g. \cite{Lio1}, \cite%
{Tem1}, \cite{Sch1}, etc.). It is known for the Euler equations were shown
that uniqueness of weak solutions isn't (see, \cite{Sch1}, \cite{Shn1}).

We need to note the regularity of solutions in three dimensional case were
investigated and partial regularity of the suitable weak solutions of the
Navier--Stokes equation were obtained (see, e.g. \cite{Sch2}, \cite%
{CafKohNir}, \cite{Lin1}, \cite{Lio1}, \cite{Lad1}, \cite{ChL-RiMay}, \cite%
{Gal}). There exist many works which study different properties of solutions
of the Navier--Stokes equation (see, \cite{Lio1}, \cite{Lad1}, \cite{Lin1}, 
\cite{Fef1}, \cite{FoiManRosTem}, \cite{FoiRosTem1}, \cite{FoiRosTem2}, \cite%
{FoiRosTem3}, \cite{GlaSveVic}, \cite{HuaWan}, \cite{PerZat}, \cite{Sol1}, 
\cite{Sol2}, \cite{Tem1}), etc.) and also different modifications of
Navier--Stokes equation (see, e.g. \cite{Lad1}, \cite{Lio1}, \cite{Sol3},
etc.). \ 

It need note that earlier under various additional conditions of the type of
certain smoothness of the weak solutions different results on the uniqueness
of solution of the incompressible Navier-Stokes equation in $3D$ case were
obtained (see, e. g. \cite{Hop}, \cite{Lad2}, \cite{Lio1}, \cite{Tem1},
etc.). Here we would like to note the result of article \cite{Fur} that
possesses of some proximity to the main result of this article. In this
article the system of equations (1.1%
${{}^1}$%
) - (1.3) was examined, which is obtained from (1.1) - (1.3) under studies
the solvability of this problem by the Hopf-Leray approach (that below will
be explained, see, e.g. \cite{Tem1}). In \cite{Fur} the problem in the
following form was studied 
\begin{equation*}
Nu=\frac{du}{dt}+\nu Au+B(u)=f,\quad \gamma _{0}u=u_{0},
\end{equation*}%
where $B(u)\equiv \underset{j=1}{\overset{3}{\sum }}u_{j}\frac{\partial u_{i}%
}{\partial x_{j}}$ and $\gamma _{0}u\equiv u\left( 0\right) $. In which the
author shows that $\left( N,\gamma _{0}\right) :Z\longrightarrow L^{2}\left(
0,T:H^{-1/2}\left( 
\Omega
\right) \right) \times H^{1/2}\left( 
\Omega
\right) $\ is the continuous operator under the condition that $%
\Omega
\subset R^{3}$ is a bounded region whose boundary $\partial 
\Omega
$ is a closed manifold of class $C^{\infty }$, where 
\begin{equation*}
Z=\left\{ \left. u\in L^{2}\left( 0,T:H^{3/2}\left( 
\Omega
\right) \right) \right\vert \ \ \frac{du}{dt}\in L^{2}\left(
0,T:H^{-1/2}\left( 
\Omega
\right) \right) \right\} .
\end{equation*}

Moreover, here is proved that if to denote by $F_{\gamma _{0}}$ the image: $%
N\left( Z_{u_{0}}\right) =F_{\gamma _{0}}$ for $u_{0}\in H^{1/2}\left( 
\Omega
\right) $ then for each $f\in F_{\gamma _{0}}$ there exists only one
solution $u\in Z$ such that $Nu=f$ and $\gamma _{0}u=u_{0}$, here $%
Z_{u_{0}}=\left\{ \left. u\in Z\right\vert \ \gamma _{0}u=u_{0}\right\} $,
and also the density of set $F_{\gamma _{0}}$ in $L^{2}\left(
0,T:H^{-1/2}\left( 
\Omega
\right) \right) $ in the topology of $L^{p}\left( 0,T:H^{-l}\left( 
\Omega
\right) \right) $ under certain conditions on $p,l$. The proof given in \cite%
{Fur} is similar to the proof of \cite{Lio1} and \cite{Tem1}, but the result
not follows from their results.

In the beginning in this paper certain explanation why for study of the
posed question is enough to investigate the problem (1.1%
${{}^1}$%
) - (\ref{3}) is provided. Here the approach Hopf-Leray (with taking into
account of the result of de Rham) for study the existence of the weak
solution of the considered problem is used, as usually.

Unlike above investigations here we study the question on the uniqueness in
the case when the weak solution $u$ of the problem (1.1%
${{}^1}$%
) - (\ref{3}) is contained in $\mathcal{V}\left( Q^{T}\right) $ (in $3D$
case), consequently, as is known, for this the following condition is
sufficiently: functions $u_{0}$ and $f$ satisfy conditions 
\begin{equation*}
u_{0}\in H\left( \Omega \right) ,\quad f\in L^{2}\left( 0,T;V^{\ast }\left(
\Omega \right) \right) .
\end{equation*}

\begin{notation}
\label{N_1}The result obtained for the problem (1.1%
${{}^1}$%
) - (\ref{3}) allows us to respond to the posed question, namely this shows
the uniqueness of the velocity vector $u$.
\end{notation}

So, in this article the uniqueness of the weak solutions $u$ obtaining by
the Hopf-Leray's approach of the mixed problem with Dirichlet boundary
condition for the incompressible Navier-Stokes system in the $3D$ case is
investigated. For investigation we use an approach that is different from
usual methods used for study of the questions of such type. The approach
used here allows us to receive more general result on the uniqueness of the
weak solution (of the velocity vector $u$) of the mixed problem for the
incompressible Navier--Stokes equation under more general conditions. In
addition, here in order to carry out of the proof of the main result, in the
beginning the existence and uniqueness of the weak solutions of auxiliary
problems are studied.

For study of the uniqueness of the solution of the problem we also use of
the formulation of the problem in the weak sense according to J. Leray \cite%
{Ler1}. As well-known, problem (\ref{1}) - (\ref{3}) and (1.1%
${{}^1}$%
) - (\ref{3}) was investigated in many works (see, \cite{Lio1}, \cite{Tem1}
and \cite{Gal}). Here we will bring the result on weak solvability from the
book \cite{Tem1}.

\begin{theorem}
(\cite{Tem1}) Let $\Omega $ be a Lipschitz open bounded set in $R^{n}$, $%
n\leq 4$. Let there be given $f$ and $u_{0}$ which

satisfy $f\in L^{2}\left( 0,T;V^{\ast }\left( \Omega \right) \right) $ and $%
u_{0}\in H\left( \Omega \right) $.\ Then there exists at least one function $%
u$ which

satisfies $u\in L^{2}\left( 0,T;V\left( \Omega \right) \right) $, $\frac{du}{%
dt}\in L^{1}\left( 0,T;V^{\ast }\left( \Omega \right) \right) $, $u\left(
0\right) =u_{0}$ and the equation 
\begin{equation}
\frac{d}{dt}\left\langle u,v\right\rangle -\left\langle \nu \Delta
u,v\right\rangle +\left\langle \underset{j=1}{\overset{n}{\sum }}u_{j}\frac{%
\partial u}{\partial x_{j}},v\right\rangle =\left\langle f,v\right\rangle
\label{1a}
\end{equation}%
for any $v\in V\left( \Omega \right) $. Moreover, $u\in L^{\infty }\left(
0,T;H\left( \Omega \right) \right) $ and $u\left( t\right) $\ is weakly
continuous from $\left[ 0,T\right] $ into $H\left( \Omega \right) $ (i. e. $%
\forall v\in H\left( \Omega \right) $, $t\longrightarrow \left\langle
u\left( t\right) ,v\right\rangle $ is a continuous scalar function, and
consequently, $\left\langle u\left( 0\right) ,v\right\rangle =\left\langle
u_{0},v\right\rangle $).
\end{theorem}

"Moreover, in the case when $n=3$ a weak solution $u$ satisfy 
\begin{equation*}
u\in V\left( Q^{T}\right) ,\quad u^{\prime }\equiv \frac{\partial u}{%
\partial t}\in L^{\frac{4}{3}}\left( 0,T;V^{\ast }(%
\Omega
)\right) ,\quad
\end{equation*}%
and also is almost everywhere (a.e.) equal to some continuous function from $%
\left[ 0,T\right] $ into $H$, so that (\ref{3}) is meaningful, with use of
the obtained properties that any weak solution belong to the bounded subset
of 
\begin{equation*}
\mathcal{V}\left( Q^{T}\right) \equiv V\left( Q^{T}\right) \cap
W^{1,4/3}\left( 0,T;V^{\ast }(%
\Omega
)\right)
\end{equation*}%
and satisfies the equation (\ref{1a})." \footnote{%
The expression $\left\langle g,h\right\rangle $ here and further denote $%
\left\langle g,h\right\rangle =\underset{i=1}{\overset{3}{\sum }}\ \underset{%
\Omega }{\int }g_{i}h_{i}dx$ for any $g,h\in \left( H\left( \Omega \right)
\right) $, or $g\in V\left( \Omega \right) $ and $h\in V^{\ast }\left(
\Omega \right) $, respectively.}

In what follows we will base on the mentioned theorem about the existence of
the weak solution of problem (1.1%
${{}^1}$%
) - (\ref{3}) and the added remarks as principal results, since here is
investigated the question related to the weak solution of the problem that
is studied in Theorem 1.

The main result of this paper is the following uniqueness theorem.

\begin{theorem}
\label{Th_1}Let $\Omega \subset R^{3}$ be a domain of $Lip_{loc}$ (will be
defined below; see, Section \ref{Sec_I.4}), $T>0$ be a number. If given
functions $u_{0}$, $f$ satisfy of conditions $u_{0}\in H^{1/2}\left( \Omega
\right) $, $f\in L^{2}\left( 0,T;H^{1/2}\left( \Omega \right) \right) $ then
weak solution $u\in \mathcal{V}\left( Q^{T}\right) $ of the problem (1.1%
${{}^1}$%
) - (\ref{3}) given by the above mentioned theorem is unique.
\end{theorem}

This article is organized as follow. In Part I the question is studied under
certain smoothness conditions onto given functions. In Section I.2 some
known results and the explanation of the relation between problems (\ref{1})
- (\ref{3}) and (1.1%
${{}^1}$%
) - (\ref{3}) is adduced, and also the necessary auxiliary results, namely
lemmas are proved. These lemmas are need us for the study of the main
problem. In Section I.3 the auxiliary problems determined that posed on the
cross-sections of $\Omega $, which are obtained from problem (1.1%
${{}^1}$%
) - (\ref{3}). Here is explained how these problems are obtained from
problem (1.1%
${{}^1}$%
) - (\ref{3}), and also is suggested to study the main question for the
auxiliary problems on the cross-sections instead of the investigation of
this question on whole of $\Omega $. In Section I.4 the existence of the
solution and, in Section I.5 the uniqueness of solution of the auxiliary
problem are studied. In Section I.6 the main result Theorem \ref{Th_1} is
proved. In Section II.7 of Part II one conditional result on uniqueness of
weak solution of problem (1.1%
${{}^1}$%
) - (\ref{3}) by use of certain modification of the well-known approach is
proved.

\part{\label{Part I}One new approach for study of the uniqueness}

\section{\label{Sec_I.2}Preliminary results}

In this section the background material, definitions of the appropriate
spaces, that will be used in the next sections are briefly recalled. In
addition, here some notations are introduced, and also the necessary
auxiliary results are proved that in the follow will be employed. Moreover,
we recall the basic setup and results regarding of the weak solutions of the
incompressible Navier--Stokes equations used throughout this paper.

As is well-known, problem (\ref{1}) - (\ref{3}) possesses weak solution in
the space $\mathcal{V}\left( Q^{T}\right) \times L^{2}\left( Q^{T}\right) $
for each $u_{0i}\left( x\right) ,$ $f_{i}(x,t)$ ($i=\overline{1,3}$), which
are contained in the suitable spaces (see, e.g. \cite{Lio1}, \cite{Tem1} and
references therein), (the space $\mathcal{V}\left( Q^{T}\right) $ will be
defined later on). Here our main problem is the investigation of the posed
question in the case $n=3$, consequently, here problems will be studied
mostly in the case $n=3$.

\begin{definition}
\label{D_2.1}Let $\Omega \subset R^{3}$ be an open bounded Lipschitz domain
and $Q^{T}\equiv \left( 0,T\right) \times \Omega $, $T>0$ be a number. Let $%
V\left( Q^{T}\right) $ be the space determined as 
\begin{equation*}
V\left( Q^{T}\right) \equiv L^{2}\left( 0,T;V\left( \Omega \right) \right) \cap L^{\infty }\left( 0,T;H\left( \Omega \right) \right) ,%
\end{equation*}%
where $V\left( \Omega \right) $ and $H\left( \Omega \right) $ are the
closure of \ 
\begin{equation*}
\left\{ \varphi \left\vert \ \varphi \in \left( C_{0}^{\infty }\left( \Omega
\right) \right) ^{3},\right. \func{div}\varphi =0\right\}
\end{equation*}%
in the topology of $\left( W_{0}^{1,2}\left( \Omega \right) \right) ^{3}$
and in the topology of $\left( L^{2}\left( \Omega \right) \right) ^{3}$,
respectively;

the dual $V\left( \Omega \right) $ determined as $V^{\ast }\left( \Omega
\right) $ and is the closure of \ the\ linear continuous functionals defined
on $V\left( \Omega \right) $ in the sense of the Lax dual relative to $%
H\left( \Omega \right) $.

Moreover we set also the space $\mathcal{V}\left( Q^{T}\right) \equiv
V\left( Q^{T}\right) \cap W^{1,4/3}\left( 0,T;V^{\ast }\left( \Omega \right)
\right) $ for $n=3$ (\cite{Tem1}).
\end{definition}

Here as is well-known $L^{2}\left( \Omega \right) $ is the Lebesgue space
and $W^{1,2}\left( \Omega \right) $ is the Sobolev space, that are the
Hilbert spaces and 
\begin{equation*}
W_{0}^{1,2}\left( \Omega \right) \equiv \left\{ v\left\vert \ v\in
W^{1,2}\left( \Omega \right) ,\right. v\left\vert \ _{\partial \Omega
}\right. =0\right\} .
\end{equation*}%
As is well-known in this case $H\left( \Omega \right) $ and $V\left( \Omega
\right) $ also are the Hilbert spaces, therefore

\begin{equation*}
V\left( \Omega \right) \subset H\left( \Omega \right) \equiv H^{\star
}\left( \Omega \right) \subset V^{\ast }\left( \Omega \right) .
\end{equation*}

So, assume the given functions $u_{0}$ and $f$ satisfy 
\begin{equation*}
u_{0}\in H\left( \Omega \right) ,\quad f\in L^{2}\left( 0,T;V^{\ast }\left(
\Omega \right) \right)
\end{equation*}%
where $V^{\ast }\left( \Omega \right) $ is the dual space of $V\left( \Omega
\right) $.

Consider the problem for which the existence of the weak solution directly
connected with the existence of the weak solution of problem (\ref{1}) - (%
\ref{3}) as will be shown below

\begin{equation}
\frac{\partial u_{i}}{\partial t}-\nu \Delta u_{i}+\underset{j=1}{\overset{n}%
{\sum }}u_{j}\frac{\partial u_{i}}{\partial x_{j}}=f_{i}\left( t,x\right)
,\quad i=\overline{1,n},\ \nu >0  \tag{1.1$^{1}$}
\end{equation}%
\begin{equation}
\func{div}u=\underset{i=1}{\overset{n}{\sum }}\frac{\partial u_{i}}{\partial
x_{i}}=\underset{i=1}{\overset{n}{\sum }}D_{i}u_{i}=0,\quad x\in \Omega
\subset R^{n},\ t>0,  \tag{1.2}
\end{equation}%
\begin{equation}
u\left( 0,x\right) =u_{0}\left( x\right) ,\quad x\in \Omega ;\quad
u\left\vert \ _{\left( 0,T\right) \times \partial \Omega }\right. =0. 
\tag{1.3}
\end{equation}%
\footnote{%
Here all equations are needed to understand in the sense of the
corresponding spaces, e.g. the equation (1.1%
${{}^1}$%
) is understood in the sense of the dual space of $\mathcal{V}\left(
Q^{T}\right) $.}

The investigation of the existence of the weak solution of problem (1.1%
${{}^1}$%
) - (\ref{3}) is equivalent to the investigation of the following equation
with corresponding initial condition (see, Theorem \ref{Th_1}) 
\begin{equation}
\frac{d}{dt}\left\langle u,v\right\rangle -\left\langle \nu \Delta
u,v\right\rangle +\left\langle \underset{j=1}{\overset{n}{\sum }}u_{j}\frac{%
\partial u}{\partial x_{j}},v\right\rangle =\left\langle f,v\right\rangle
\label{1b}
\end{equation}%
where $v\in V(%
\Omega
)$ is arbitrary.

In other words, one must study the existence of the weak solution of problem
(1.1%
${{}^1}$%
) - (\ref{3}) in the sense of J. Leray \cite{Ler1} by use of his approach
(see, also \cite{Lio1}, \cite{Tem1}). This approach shows that for study of
the uniqueness of the solution relative to velosity vector $u$ of problem (%
\ref{1}) - (\ref{3}) sufficiently to investigate of same question for
problem (1.1%
${{}^1}$%
) - (\ref{3}) in view of de Rham result (see, books \cite{Lio1}, \cite{Tem1}%
, \cite{Lad1}, \cite{FoiManRosTem}, \cite{Gal}, \cite{MajBer} etc. where the
properties of the this problem were explained enough clearly).

In order to explain that the investigation of the posed question for problem
(1.1%
${{}^1}$%
) - (\ref{3}) is sufficient for our goal we represent here some results of
the book \cite{Tem1}, which have the immediate relation to this problem.

\begin{proposition}
\label{Pr_2.1}(\cite{Tem1}) Let $\Omega $ be a bounded Lipschitz open set in 
$R^{n}$ and $g=\left( g_{1},...,g_{n}\right) $, $g_{i}\in \mathcal{D}%
^{\prime }\left( \Omega \right) $, $1\leq i\leq n$. A necessary and
sufficient condition that $g=\func{grad}p$ for some $p$ in $\mathcal{D}%
^{\prime }\left( \Omega \right) $, is that $\left\langle g,v\right\rangle =0$
$\forall v\in V\left( \Omega \right) $.
\end{proposition}

\begin{proposition}
\label{Pr_2.2}(\cite{Tem1}) Let $\Omega $ be a bounded Lipschitz open set in 
$R^{n}$.

(i) If a distribution $p$ has all its first-order derivatives $D_{i}p$, $%
1\leq i\leq n$ in $L^{2}\left( \Omega \right) $, then $p\in L^{2}\left(
\Omega \right) $ and\ 
\begin{equation*}
\left\Vert p\right\Vert _{L^{2}\left( \Omega \right) /R}\leq c\left( \Omega
\right) \left\Vert \func{grad}p\right\Vert _{\left( L^{2}\left( \Omega
\right) \right) ^{3}};
\end{equation*}

(ii) If a distribution $p$ has all its first derivatives $D_{i}p$, $1\leq
i\leq d$ in $H^{-1}\left( \Omega \right) $, then $p\in L^{2}\left( \Omega
\right) $ and 
\begin{equation*}
\left\Vert p\right\Vert _{L^{2}\left( \Omega \right) /R}\leq c\left( \Omega
\right) \left\Vert \func{grad}p\right\Vert _{H^{-1}\left( \Omega \right) }.
\end{equation*}

In both cases, if $\Omega $ is any open set in $R^{n}$, then $p\in
L_{loc}^{2}\left( \Omega \right) $.
\end{proposition}

Combining these results, one can note that if $g\in H^{-1}\left( \Omega
\right) $ (or $g\in L^{2}\left( \Omega \right) $) and $(g,v)=0$, then $g=%
\func{grad}p$ with $p\in L^{2}\left( \Omega \right) $ (or $p\in H^{1}\left(
\Omega \right) $) if $\Omega $ is a Lipschitz open bounded set.

\begin{theorem}
\label{Th_1.2}(\cite{Tem1}) Let $\Omega $ be a Lipschitz open bounded set in 
$R^{n}$. Then 
\begin{equation*}
H^{\bot }\left( \Omega \right) =\left\{ w\in \left( L^{2}\left( \Omega
\right) \right) ^{n}\ \left\vert \ w=\func{grad}p,\ p\in H^{1}\left( \Omega
\right) \right. \right\} ;
\end{equation*}%
\begin{equation*}
H\left( \Omega \right) =\left\{ u\in \left( L^{2}\left( \Omega \right)
\right) ^{n}\ \left\vert \func{div}u=0,\ u_{\partial \Omega }\ =0\ \right.
\right\} .
\end{equation*}
\end{theorem}

\begin{lemma}
\label{L_1.1}(see, e.g. \cite{Lio1}, \cite{Tem1} and also,\cite{Sol2}, \cite%
{Sol4} ) Let $B_{0},B,B_{1}$ be three Banach spaces, each space continuously
included in the following one $B_{0}\subset B\subset B_{1}$ and $B_{0},B_{1}$
are reflexive, moreover, the inclusion $B_{0}\subset B$ is compacts.

Let $X$ be 
\begin{equation*}
X\equiv \left\{ u\left\vert \ u\in L^{p_{0}}(0,T;B_{0}),\right. u\prime \in
L^{p_{1}}(0,T;B_{1})\right\} ,
\end{equation*}%
where $1<p_{j}<\infty $, $j=0,1$ and $0<T<\infty $. Hence $X$\ is Banach
space with the norm 
\begin{equation*}
\left\Vert u\right\Vert _{X}=\left\Vert u\right\Vert
_{L^{p_{0}}(0,T;B_{0})}+\left\Vert u\right\Vert _{L^{p_{1}}(0,T;B_{1})}.
\end{equation*}%
Then under these conditons the inclusion $X\subset L^{p_{0}}(0,T;B)$ is
compact.

Moreover, the inclusion $X\subset C(0,T;B_{1})$ holds, due of Lebesgue
theorem.
\end{lemma}

Consequently, if one will seek of weak solution of the problem (\ref{1}) - (%
\ref{3}) by according Hopf-Leray then one can get the following equation 
\begin{equation}
\frac{d}{dt}\left\langle u,v\right\rangle -\left\langle \nu \Delta
u,v\right\rangle +\left\langle \underset{j=1}{\overset{n}{\sum }}u_{j}\frac{%
\partial u}{\partial x_{j}},v\right\rangle =\left\langle f,v\right\rangle
-\left\langle \nabla p,v\right\rangle ,  \label{2.1}
\end{equation}%
where $v\in V(%
\Omega
)$ is arbitrary.

So, if to consider of the last adding in the right side then at illumination
of above results (Propositions \ref{Pr_2.1}, \ref{Pr_2.2} and Theorem \ref%
{Th_1.2}) using the integration by parts and taking into account that $v\in
V(%
\Omega
)$, i.e. $\func{div}v=0$ and $v\left\vert \ _{\left( 0,T\right) \times
\partial \Omega }\right. =0$ we will get the equation 
\begin{equation}
\left\langle \nabla p,v\right\rangle \equiv \underset{\Omega }{\int }\nabla
p\cdot v\ dx=\underset{\Omega }{\int }p\func{div}v\ dx=0,\quad \forall v\in
V(%
\Omega
)  \label{2.2}
\end{equation}%
by virtue of de Rham result. Consequently, taking into account (\ref{2.2})
in (\ref{2.1}) we obtain equation (\ref{1b}) that shows why for study of the
posed question is enough to study this question for problem (1.1%
${{}^1}$%
) - (\ref{3}).

So, we can continue the investigation of the posed question for problem (1.1%
${{}^1}$%
) - (\ref{3}) in the case $n=3$.

Let $\Omega \subset R^{3}$ be a bounded open domain with the boundary $%
\partial \Omega $ of the Lipschitz class. We will denote by $H^{1/2}\left(
\Omega \right) $ the vector space defined as in Definition \ref{D_2.1} by 
\begin{equation*}
\left( W^{1/2}\left( \Omega \right) \right) ^{3}\equiv \left\{ w\left\vert \
w_{i}\in W^{1/2}\left( \Omega \right) ,\right. i=1,2,3\right\} ,\quad
w=\left( w_{1},w_{2},w_{3}\right)
\end{equation*}%
where $W^{1/2,2}\left( \Omega \right) $ is the Sobolev-Slobodeskij space
(see, \cite{LioMag}, etc.). As well-known the trace for the function of the
space $H^{1/2}\left( \Omega \right) $ is definite for each smooth surface
from $\Omega $ (see, e.g. \cite{LioMag}, \cite{BesIlNik} and references
therein), which is necessary for application of our approach to the
considered problem. The main theorem will be proved under this additional
condition that is the sufficient condition for present investigation.

\begin{definition}
\label{D_2.2}A $u\in \mathcal{V}\left( Q^{T}\right) $ is called a solution
of problem (1.1%
${{}^1}$%
) - (\ref{3}) if $u\left( t\right) $ almost everywhere in $\left( 0,T\right) 
$ satisfies the following equation 
\begin{equation}
\frac{d}{dt}\left\langle u,v\right\rangle -\left\langle \nu \Delta
u,v\right\rangle +\left\langle \underset{j=1}{\overset{n}{\sum }}%
u_{j}D_{j}u,v\right\rangle =\left\langle f,v\right\rangle  \label{2.3a}
\end{equation}%
for any $v\in V\left( \Omega \right) $ and $u\left( t\right) $\ is weakly
continuous from $\left[ 0,T\right] $ into $H\left( \Omega \right) $ (i.e. $%
\forall v\in H\left( \Omega \right) $, $t\longrightarrow \left\langle
u\left( t\right) ,v\right\rangle $ is a continuous scalar function, and
consequently, $\left\langle u\left( 0\right) ,v\right\rangle =\left\langle
u_{0},v\right\rangle $).
\end{definition}

In what follows we will understand of an existing solutions be functions
that satisfy this definition together with the standard notations that are
used usually. Moreover, as above were noted if $\Omega $ be a Lipschitz open
bounded set in $R^{3}$, functions $f$ and $u_{0}$ satisfy $f\in L^{2}\left(
0,T;V^{\ast }\left( \Omega \right) \right) $ and $u_{0}\in H\left( \Omega
\right) $, respectively, then the vector function $u$ is the solution of
problem (1.1%
${{}^1}$%
) - (\ref{3}) if it satisfies of conditions of Definition \ref{D_2.2}, in
addition, $u\in L^{\infty }\left( 0,T;H\left( \Omega \right) \right) $ and
the term $\underset{j=1}{\overset{3}{\sum }}u_{j}D_{j}u\equiv B\left(
u\right) $ belong to $L^{4/3}\left( 0,T;V^{\ast }\left( \Omega \right)
\right) $.

Now we go over into main question: let problem (1.1%
${{}^1}$%
) - (\ref{3}) have two different solutions $u,v\in \mathcal{V}\left(
Q^{T}\right) $ then within the known approach we derive that the function $%
w(t,x)=u(t,x)-v(t,x)$ must satisfies the following problem 
\begin{equation}
\frac{1}{2}\frac{\partial }{\partial t}\left\Vert w\right\Vert _{2}^{2}+\nu \left\Vert \nabla w\right\Vert _{2}^{2}+\underset{j,k=1}{\overset{3}{\sum }}\left\langle \frac{\partial v_{k}}{\partial x_{j}}w_{k},w_{j}\right\rangle =0,
\label{2.3}
\end{equation}%
\begin{equation}
w\left( 0,x\right) \equiv w_{0}\left( x\right) =0,\quad x\in \Omega ;\quad w\left\vert \ _{\left( 0,T\right) \times \partial \Omega }\right. =0,
\label{2.4}
\end{equation}%
where $\left\langle g,h\right\rangle =\underset{i=1}{\overset{3}{\sum }}%
\underset{\Omega }{\int }g_{i}h_{i}dx$ for any $g,h\in \left( H\left( \Omega
\right) \right) ^{3}$, or $g\in V\left( \Omega \right) $ and $h\in V^{\ast
}\left( \Omega \right) $, respectively. So, for the proof of the uniqueness
of solution it is follows to show that $w\equiv 0$ in the sense of needed
space.

Later in this section one will studied questions and provided certain
results that are necessary for employing of the basic approach to study of
the requered question. More exactly, these reasonings and results will be
used in sections 3-6 for study of the posed question.

As our purpose is the investigation of the uniqueness of solution of problem
(1.1%
${{}^1}$%
) - (\ref{3}) therefore we will go over to the discussion of this question.
As is known, problem (1.1%
${{}^1}$%
) - (\ref{3}) has weak solution $u\left( t\right) $ from the space $\mathcal{%
V}\left( Q^{T}\right) $ denoted in Definition \ref{D_2.1}, which possesses
of the above mentioned properties and also some complementary properties of
the smoothness type (see, \cite{Ler1}, \cite{Lad2}, \cite{Lio1}, \cite%
{CafKohNir}, \cite{ConKukVic}, \cite{FoiManRosTem}, \cite{ChL-RiMay}).
Therefore we will conduct our study under the condition that problem (1.1%
${{}^1}$%
) - (\ref{3}) have weak solutions and they belong to $\mathcal{V}\left(
Q^{T}\right) $. For the study of the uniqueness of solution of problem (1.1%
${{}^1}$%
) - (\ref{3})) as above assume that problem (1.1%
${{}^1}$%
) - (\ref{3})) has, at least, two different solutions $u,v\in \mathcal{V}%
\left( Q^{T}\right) $. But for demonstrate that this isn%
\'{}%
t possible we will employ a different procedure.

Consequently, if we assume that problem (1.1%
${{}^1}$%
) - (\ref{3}) have two different solutions then they need to be different at
least on some subdomain $Q_{1}^{T}$ of $Q^{T}$. In other words there exists
a subdomain $\Omega _{1}$ of $\Omega $ and an interval $\left(
t_{1},t_{2}\right) \subseteq \left( 0,T\right] $ such that 
\begin{equation*}
Q_{1}^{T}\subseteq \left( t_{1},t_{2}\right) \times \Omega _{1}\subseteq
Q^{T}
\end{equation*}%
with $mes_{4}\left( Q_{1}^{T}\right) >0$ for which 
\begin{equation}
mes_{4}\left( \left\{ (t,x)\in Q^{T}\left\vert \ \left\vert u(t,x)-v(t,x)\right\vert \right. >0\right\} \right) =mes_{4}\left( Q_{1}^{T}\right) >0
\label{2.5}
\end{equation}%
holds, where $mes_{4}\left( Q_{1}^{T}\right) $ denote the measure of $%
Q_{1}^{T}$ in $R^{4}$ (i.e. $mes_{k}$ denote the Lebesgue measure on $k$
dimensional space $R^{n}$). Whence follows, that subdomain $\Omega _{1}$
must have of the positive Lebesgue measure, i.e. $mes_{3}(\Omega _{1})>0$.

The following lemmas will proved even though for $n>1$, but mostly these
will use for the case $n=4$.

So, it is need to prove the following lemmas, which will use later on.

\begin{lemma}
\label{L_2.1}Let $G\subset R^{n}$ be Lebesgue measurable subset then the
following statements are equivalent:

1) $\infty >mes_{n}\left( G\right) >0;$

2) there exist a subsets $I\subset R^{1}$, $\infty >mes_{1}\left( I\right)
>0 $ and $G_{\beta }\subset L_{\beta ,n-1}$, $\infty >mes_{n-1}\left(
G_{\beta }\right) >0$ such that $G=\underset{\beta \in I}{\cup }G_{\beta
}\cup N$, where $N$ is a set with $mes_{n-1}\left( N\right) =0$, and $%
L_{\beta ,n-1}$ is the hyperplane of $R^{n}$, with $\ \ \ co\dim
_{n}L_{\beta ,n-1}=1$, for any $\beta \in I$, which is generated by the
arbitrary fixed vector $y_{0}\in R^{n}$ and defined as follow 
\begin{equation*}
L_{\beta ,n-1}\equiv \left\{ y\in R^{n}\left\vert \ \left\langle
y_{0},y\right\rangle =\beta \right. \right\} ,\quad \forall \beta \in I.
\end{equation*}
\end{lemma}

\begin{proof}
Let $mes_{d}\left( G\right) >0$ and consider the class of hyperplanes $%
L_{\gamma ,n-1}$ for which $G\cap L_{\gamma ,n-1}\neq \varnothing $ and $%
\gamma \in I_{1}$, where $I_{1}\subset R^{1}$\ be some subset. It is clear
that 
\begin{equation*}
G\equiv \underset{\gamma \in I_{1}}{\bigcup }\left\{ x\in G\cap L_{\gamma
,n-1}\left\vert \ \gamma \in I_{1}\right. \right\} .
\end{equation*}%
Then there exists a subclass of hyperplanes $\left\{ L_{\gamma
,n-1}\left\vert \ \gamma \in I_{1}\right. \right\} $ for which

$mes_{n-1}\left( G\cap L_{\gamma ,n-1}\right) >0$ is fulfilled. The number
of such type hyperplanes cannot be less than countable or equal it because $%
mes_{n}\left( G\right) >0$, moreover this subclass of $I_{1}$\ must possess
the $R^{1}$ measure greater than $0$ since $mes_{n}\left( G\right) >0$.
Indeed, let $I_{1,0}$ be this subclass and $mes_{1}\left( I_{1,0}\right) =0$%
. In this case we get subset 
\begin{equation*}
\left\{ \left( \gamma ,y\right) \in I_{1,0}\times G\cap L_{\gamma
,n-1}\left\vert \ \gamma \in I_{1,0},y\in G\cap L_{\gamma ,n-1}\right.
\right\} \subset R^{n}
\end{equation*}%
where $mes_{n-1}\left( G\cap L_{\gamma ,n-1}\right) >0$ for all $\gamma \in
I_{1,0}$, but $mes_{1}\left( I_{1,0}\right) =0$, then 
\begin{equation*}
mes_{n}\left( \left\{ \left( \gamma ,y\right) \in I_{1,0}\times G\cap
L_{\gamma ,n-1}\left\vert \ \gamma \in I_{1,0}\right. \right\} \right) =0.
\end{equation*}%
On the other hand if $mes_{n-1}\left( G\cap L_{\gamma ,n-1}\right) =0$ for
all $\gamma \in I_{1}-I_{1,0}$ then 
\begin{equation*}
mes_{n}\left( \left\{ \left( \gamma ,y\right) \in I_{1}\times G\cap
L_{\gamma ,n-1}\left\vert \ \gamma \in I_{1}\right. \right\} \right) =0,
\end{equation*}%
whence follows 
\begin{equation*}
mes_{n}\left( G\right) =mes_{n}\left( \left\{ \left( \gamma ,y\right) \in
I_{1}\times G\cap L_{\gamma ,n-1}\left\vert \ \gamma \in I_{1}\right.
\right\} \right) =0.
\end{equation*}

But this contradicts the condition $mes_{n}\left( G\right) >0$.
Consequently, the statement 2 holds.

Let the statement 2 holds. It is clear that the class of hyperplanes $%
L_{\beta ,n-1}$ defined by such way are parallel and also we can define the
class of subsets of $G$ as its cross-section with hyperplanes, i.e. in the
form: $G_{\beta }\equiv G\cap L_{\beta ,n-1}$, \ \ $\beta \in I$. Then $%
G_{\beta }\neq \varnothing $ and we can write $G_{\beta }\equiv G\cap
L_{\beta ,n-1}$, $\beta \in I$, moreover $G\equiv \underset{\beta \in I}{%
\bigcup }\left\{ x\in G\cap L_{\beta ,n-1}\left\vert \ \beta \in I\right.
\right\} \cup N$. Whence we get 
\begin{equation*}
G\equiv \left\{ \left( \beta ,x\right) \in I\times G\cap L_{\beta
,n-1}\left\vert \ \beta \in I,x\in G\cap L_{\beta ,n-1}\right. \right\} \cup
N.
\end{equation*}

Consequently, $mes_{n}\left( G\right) >0$ by virtue of conditions: $%
mes_{1}\left( I\right) >0$ and

$mes_{n-1}\left( G_{\beta }\right) >0$ for any $\beta \in I$.
\end{proof}

Lemma \ref{L_2.1} shows that for the study of the measure of some subset $%
\Omega
\subseteq R^{n}$ it is enough to study its stratifications by a class of
corresponding hyperplanes.

\begin{lemma}
\label{L_2.2}Let problem (1.1%
${{}^1}$%
) - (\ref{3}) has, at least, two different solutions $u,v$ that are
contained in $\mathcal{V}\left( Q^{T}\right) $ and assume that $%
Q_{1}^{T}\subseteq Q^{T}$ is one of a subdomain of $Q^{T}$ where $u$ and $v$
are different. Then there exists, at least, one class of parallel
hyperplanes $L_{\alpha }$, $\alpha \in I\subseteq \left( \alpha _{1},\alpha
_{2}\right) \subset R^{1}$ ($\alpha _{2}>\alpha _{1}$)\ with $co\dim
_{R^{3}}L_{\alpha }=1$ such, that $u\neq v$ on $Q_{L_{\alpha }}^{T}\equiv %
\left[ \left( 0,T\right) \times \left( \Omega \cap L_{\alpha }\right) \right]
\cap Q_{1}^{T}$, and vice versa, here $mes_{1}\left( I\right) >0$, $%
mes_{2}\left( \Omega \cap L_{\alpha }\right) >0$ and $L_{\alpha }$ are
hyperplanes which are defined as follows: there is vector $x_{0}\in
S_{1}^{R^{3}}\left( 0\right) $ such that 
\begin{equation*}
L_{\alpha }\equiv \left\{ x\in R^{3}\left\vert \ \left\langle
x_{0},x\right\rangle =\alpha ,\right. \ \forall \alpha \in I\right\} .
\end{equation*}
\end{lemma}

\begin{proof}
Let problem (1.1%
${{}^1}$%
) - (\ref{3}) have two different solutions $u,v\in \mathcal{V}\left(
Q^{T}\right) $ then there exist a subdomain of $Q^{T}$ on which these
solutions are different. Then there are $t_{1},t_{2}>0$ such that 
\begin{equation}
mes_{3}\left( \left\{ x\in \Omega \left\vert \ \left\vert u\left( t,x\right) -v\left( t,x\right) \right\vert >0\right. \right\} \right) >0
\label{2.6}
\end{equation}%
holds for any $t\in J\subseteq \left[ t_{1},t_{2}\right] \subseteq \left[
0,T\right) $, where $mes_{1}\left( J\right) >0$ by the virtue of the
condition 
\begin{equation*}
mes_{4}\left( \left\{ (t,x)\in Q^{T}\left\vert \ \left\vert
u(t,x)-v(t,x)\right\vert \right. >0\right\} \right) >0
\end{equation*}%
and of Lemma \ref{L_2.1}.

Whence follows, that there exist, at least, one class of the parallel
hyperplanes $L_{\alpha }$, $\alpha \in I\subseteq \left( \alpha _{1},\alpha
_{2}\right) \subset R^{1}$ such that $co\dim _{R^{3}}L_{\alpha }=1$ and 
\begin{equation}
mes_{2}\left( \left\{ x\in \Omega \cap L_{\alpha }\left\vert \ \left\vert u\left( t,x\right) -v\left( t,x\right) \right\vert >0\right. \right\} \right) >0,\ \forall \alpha \in I
\label{2.7}
\end{equation}%
hold for $\forall t\in J$, where subsets $I$ and $J$ are satisfy
inequations: $mes_{1}\left( I\right) >0$, $mes_{1}\left( J\right) >0$, and
also (\ref{2.7}) holds, by virtue of (\ref{2.6}). This proves the "if" part
of Lemma.

Now consider the converse assertion. Let there exist a class of hyperplanes $%
L_{\alpha }$, $\alpha \in I_{1}\subseteq \left( \alpha _{1},\alpha
_{2}\right) \subset R^{1}$ with $co\dim _{R^{3}}L_{\alpha }=1$ that fulfills
the condition of Lemma and the subset $I_{1}$\ satisfies of same condition
as $I$. Then there exist, at least, one subset $J_{1}$ of $\left[ 0,T\right) 
$ such that $mes_{1}\left( J_{1}\right) >0$ and the inequation $u\left(
t,x\right) \neq v\left( t,x\right) $ holds onto $Q_{2}^{T}$ with $%
mes_{4}\left( Q_{2}^{T}\right) >0$, which is defined as $Q_{2}^{T}\equiv
J_{1}\times U_{L}$, where 
\begin{equation}
U_{L}\equiv \underset{\alpha \in I_{1}}{\bigcup }\left\{ x\in \Omega \cap L_{\alpha }\left\vert \ u\left( t,x\right) \neq v\left( t,x\right) \right. \right\} \subset \Omega ,\ t\in J_{1}
\label{2.8}
\end{equation}%
for which the inequation $mes_{R^{3}}\left( U_{L}\right) >0$ is fulfilled by
virtue of the condition and of Lemma \ref{L_2.1}.

So we get 
\begin{equation*}
u\left( t,x\right) \neq v\left( t,x\right) \text{ onto }Q_{2}^{T}\equiv
J_{1}\times U_{L},\text{ with }mes_{4}\left( Q_{2}^{T}\right) >0.
\end{equation*}%
Thus, we obtain the fact that $u\left( t,x\right) $ and $v\left( t,x\right) $
are different functions in $\mathcal{V}\left( Q^{T}\right) $.
\end{proof}

It is not difficult to see that result of Lemma \ref{L_2.2} is independent
of assumption: $Q_{1}^{T}\subset Q^{T}$ or $Q_{1}^{T}=Q^{T}$.

Likely one could be to prove more general results of such type using of the
regularity properties of weak solutions of this problem (see, \cite{Sch1}, 
\cite{CafKohNir}, \cite{Lin1}, \cite{ChL-RiMay}, etc.).

\section{\label{Sec_I.3}Investigation of the auxiliary problem}

In this section we will transform problem (1.1%
${{}^1}$%
) - (\ref{3}) to the auxiliary problems in order to use of the result of
Lemma \ref{L_2.2}. In other words, here our concept of the investigation of
the posed question will presented. This concept is based to result of Lemma %
\ref{L_2.2}, which shows, that for study of posed problem it is enough to
investigate this problem on the cross-sections of the domain $Q^{T}\equiv
\left( 0,T\right) \times \Omega $.

So, we will begin with the definition of the domain \textit{\ }$%
\Omega
\subset R%
{{}^3}%
$ on which will be study of the problem.

\begin{definition}
\label{D_3.1}\textit{A bounded open domain }$%
\Omega
\subset R%
{{}^3}%
$\textit{\ with the boundary }$\partial 
\Omega
$\textit{\ is spoken from the class }$Lip_{loc}$\textit{\ iff }$\partial 
\Omega
$\textit{\ is a locally Lipschitz hypersurface. (This means: any point }$%
x\in \partial 
\Omega
$\textit{\ possesses a neighbourhood in }$\partial 
\Omega
$\textit{\ that admits a }representation as a hypersurface $y_{3}=\psi
\left( y_{1},y_{2}\right) $, where $\psi $ is a Lipschitz function, and $%
\left( y_{1},y_{2},y_{3}\right) $ are rectangular coordinates in $R%
{{}^3}%
$. In a coordinate basis that may be different from the canonical basis $%
\left( e_{1},e_{1},e_{3}\right) $.)
\end{definition}

According to $\Omega $\textit{\ }is a locally Lipschitz and bounded one can
draw the conclusion: each point $x_{j}\in \partial \Omega $, has an open
neighbourhood $U_{j}$ such that $U_{j}^{\prime }=\Omega \cap U_{j}$ ,
moreover, $\partial \Omega $ can be covered by a finite family of such sets $%
U_{j}^{\prime }$, $j\in J$, that boundary $U_{j}^{\prime }$, $j\in J$ is
Lipschitz, or $\partial \Omega \in Lip_{loc}$. Consequently\textit{\ }for
every "cross-section" $\Omega _{L}\equiv \Omega \cap L\neq \varnothing $\ of 
$%
\Omega
$\ with arbitrary hyperplain $L$\ exists, at least, one coordinate subspace (%
$\left( x_{j},x_{k}\right) $)\ which possesses a domain $P_{x_{i}}\Omega
_{L} $\ (or union of domains) whit the Lipschitz class boundary since $%
\partial \Omega _{L}\equiv \partial \Omega \cap L\neq \varnothing $ and
isomorphically defined of $\Omega _{L}$\ with the affine representation,%
\textit{\ }in addition $\partial \Omega _{L}\Longleftrightarrow \partial
P_{x_{i}}\Omega _{L}$.

Thus, with use of the representation $P_{x_{i}}L$ of the hyperplane $L$ we
get that $\Omega _{L}$ can be written in the form $P_{x_{i}}\Omega _{L}$,
therefore an integral on $\Omega _{L}$ also will defined by the respective
representation, i. e. as the integral on $P_{x_{i}}\Omega _{L}$.

It should be noted that $\Omega _{L}$ can consist of many parts then $%
P_{x_{i}}\Omega _{L}$ will be such as $\Omega _{L}$. Consequently in this
case $\Omega _{L}$ will be as the union of domains and the following
relation will be holds 
\begin{equation*}
\Omega _{L}=\underset{r=1}{\overset{m}{\cup }}\Omega _{L}^{r}\
\Longleftarrow \Longrightarrow \ P_{x_{i}}\Omega _{L}=\underset{r=1}{\overset%
{m}{\cup }}P_{x_{i}}\Omega _{L}^{r},\quad \infty >m\geq 1,
\end{equation*}%
by virtue of the definition \ref{D_3.1}. Therefore, each of $P_{x_{i}}\Omega
_{L}^{r}$ will be the domain and one can investigate these separately, as
the following inclusions take place: $\Omega _{L}^{r}\subset \Omega $ and $%
\partial \Omega _{L}^{r}\subset \partial \Omega $.

So, we will define subdomains of $Q^{T}\equiv \left( 0,T\right) \times
\Omega $ as follows $Q_{L}^{T}\equiv \left( 0,T\right) \times \left( \Omega
\cap L\right) $, where $L$ is arbitrary fixed hyperplane of the dimension
two and $\Omega \cap L\neq \varnothing $. Therefore, we will study the
problem onto the subdomain defined by use of the "cross-section" of $\Omega $
whit arbitrary fixed hyperplane of the dimension two $L$, i.e. the $co\dim
_{R^{3}}L=1$ ($\Omega \cap L$, namely on $Q_{L}^{T}\equiv \left( 0,T\right)
\times \left( \Omega \cap L\right) $).

Consequently, we will investigate uniqueness of the problem (1.1%
${{}^1}$%
) - (\ref{3}) on the "cross-section" $Q_{L}^{T}$ defined by the
"cross-section" of $\Omega $, where $\Omega \subset R^{3}$. This
"cross-section" is understood in the following sense: Let $L$ be a
hyperplane in $R^{3}$ with $co\dim _{R^{3}}L=1$, clearly that $L$ is certain
shift of $R^{2}$ or $R^{2}$. Denote by $\Omega _{L}$ of the "cross-section" $%
\Omega _{L}\equiv \Omega \cap L\neq \varnothing $, $mes_{R^{2}}\left( \Omega
_{L}\right) >0$, e.g. $L$ can be $L\equiv \left\{ \left(
x_{1},x_{2},0\right) \left\vert \ x_{1},x_{2}\in R^{1}\right. \right\} $. In
other words, if $L$ is the hyperplane in $R^{3}$ then we can determine it as 
\begin{equation*}
L\equiv \left\{ x\in R^{3}\left\vert \ \left\langle a,x\right\rangle
=a_{1}x_{1}+a_{2}x_{2}+a_{3}x_{3}=b\right. \right\} ,
\end{equation*}%
where $a\in S_{1}^{R^{3}}\left( 0\right) $ is arbitrary fixed unit vector of 
$R^{3}$and $b\in R^{1}$ is arbitrary fixed constant, furthermore each $a\in
S_{1}^{R^{3}}\left( 0\right) $ and $b\in R^{1}$ define of single $%
L_{b}\left( a\right) $ and vice versa. Whence follows $%
a_{3}x_{3}=b-a_{1}x_{1}-a_{2}x_{2}$, if we assume $a_{3}\neq 0$ then $x_{3}=%
\frac{1}{a_{3}}\left( b-a_{1}x_{1}-a_{2}x_{2}\right) $ , moreover, if we
takes of substitutions: $\frac{b}{a_{3}}\Longrightarrow b,\frac{a_{1}}{a_{3}}%
\Longrightarrow a_{1}$ and $\frac{a_{2}}{a_{3}}\Longrightarrow a_{2}$ then
we derive $x_{3}\equiv \psi _{3}\left( x_{1},x_{2}\right)
=b-a_{1}x_{1}-a_{2}x_{2}$ in the new coefficients.

Since we will investigate the problem (1.1$^{1}$)-(\ref{3}) on $Q_{L}^{T}$,
in the beginning we need define the problem that we will derive after using
this projection to the problem (1.1$^{1}$)-(\ref{3}). In other words, if we
denote by $F:D\left( F\right) \subseteq V\left( Q^{T}\right) \longrightarrow
L^{2}\left( 0,T;V^{\ast }\left( \Omega \right) \right) \times L^{2}\left(
\Omega \right) $ the operator generated by problem (1.1$^{1}$)-(\ref{3}),
then we must determine of the derived problem after projection of the
operator $F$ on $Q_{L}^{T}$. Clearly under this projection some of the
expressions in the problem (1.1$^{1}$)-(\ref{3}) will change according of
above relation, and we will derive the problem that we need to study.
Consequently, now we will derive these expressions.

Thus, we get 
\begin{equation}
D_{3}\equiv \frac{\partial x_{1}}{\partial x_{3}}D_{1}+\frac{\partial x_{2}}{\partial x_{3}}D_{2}=-a_{1}^{-1}D_{1}-a_{2}^{-1}D_{2}\quad \&
\label{3.1}
\end{equation}%
\begin{equation}
D_{3}^{2}=a_{1}^{-2}D_{1}^{2}+a_{2}^{-2}D_{2}^{2}+2a_{1}^{-1}a_{2}^{-1}D_{1}D_{2},\quad D_{i}=\frac{\partial }{\partial x_{i}},i=1,2,3,
\label{3.2}
\end{equation}%
according to above mentioned reasoning.

We will assume that functions $u_{0}$ and $f$ satisfy of conditions of
Theorem \ref{Th_1} , namely $u_{0}\in H^{1/2}\left( \Omega \right) $, $f\in
L^{2}\left( 0,T;H^{1/2}\left( \Omega \right) \right) $ that are needed for
the application of our approach. Consequently, functions $u_{0}$ and $f$\
are correctly defined on $\left( 0,T\right] \times \Omega _{L}$.

Let $L$ be arbitrary hyperplane intersecting with $\Omega $, i.e. $\Omega
_{L}\neq \varnothing $ and $u\in \mathcal{V}\left( Q^{T}\right) $ is the
solution of the problem (1.1%
${{}^1}$%
) - (\ref{3}). We will be investigate of the posed question according of
Lemma \ref{L_2.2}. More precisely, we will study of the posed question for
the problem generated by the "projection" (or "trace") of problem (1.1%
${{}^1}$%
) - (\ref{3}) onto $\left( 0,T\right] \times \Omega _{L}$.

So, we would like to apply of Lemma \ref{L_2.2} to solutions of the problem
(1.1%
${{}^1}$%
) - (\ref{3}), for that it is necessary to study of properties of solutions
of the problem (1.1%
${{}^1}$%
) - (\ref{3}) in "cross-section" $\left( 0,T\right] \times \Omega _{L}$.
Consequently, one need to study the problem which is received from the
problem (1.1%
${{}^1}$%
) - (\ref{3}) by "projection" (or "trace") it to $\left( 0,T\right] \times
\Omega _{L}$ in order to investigate of the needed properties of solutions
of the problem (1.1%
${{}^1}$%
) - (\ref{3}) on $\left( 0,T\right] \times \Omega _{L}$.

As function $u$ belong to $\mathcal{V}\left( Q^{T}\right) $, therefore the
function $u$ on $\left( 0,T\right] \times \Omega _{L}$ is well defined.
Thus, we obtain the following problem on $\left( 0,T\right] \times \Omega
_{L}$ 
\begin{equation*}
\frac{\partial u}{\partial t}-\nu \Delta u+\underset{j=1}{\overset{3}{\sum }}%
u_{j}D_{j}u=\frac{\partial u_{L}}{\partial t}-\nu \left(
D_{1}^{2}+D_{2}^{2}+D_{3}^{2}\right) u_{L}+
\end{equation*}%
\begin{equation*}
u_{L1}D_{1}u_{L}+u_{L2}D_{2}u_{L}+u_{L3}D_{3}u_{L}=\frac{\partial u_{L}}{%
\partial t}-\nu \left[ D_{1}^{2}+D_{2}^{2}+a_{1}^{-2}D_{1}^{2}\right. +
\end{equation*}%
\begin{equation*}
\left. a_{2}^{-2}D_{2}^{2}+2a_{1}^{-1}a_{2}^{-1}D_{1}D_{2}\right]
u_{L}+u_{L1}D_{1}u_{L}+u_{L2}D_{2}u_{L}-u_{L3}a_{1}^{-1}D_{1}u_{L}-
\end{equation*}%
\begin{equation*}
u_{L3}a_{2}^{-1}D_{2}u_{L}=\frac{\partial u_{L}}{\partial t}-\nu \left[
\left( 1+a_{1}^{-2}\right) D_{1}^{2}+\left( 1+a_{2}^{-2}\right) D_{2}^{2}%
\right] u_{L}-
\end{equation*}%
\begin{equation}
2\nu a_{1}^{-1}a_{2}^{-1}D_{1}D_{2}u_{L}+\left(
u_{L1}-a_{1}^{-1}u_{L3}\right) D_{1}u_{L}+\left(
u_{L2}-a_{2}^{-1}u_{L3}\right) D_{2}u_{L}=f_{L}  \label{3.3}
\end{equation}%
on $\left( 0,T\right) \times \Omega _{L}$, by virtue of the above reasons,
of the conditions of the main theorem, and also of the presentations (\ref%
{3.1}) and (\ref{3.2}). We get 
\begin{equation}
\func{div}u_{L}=D_{1}\left( u_{L}-a_{1}^{-1}u_{L3}\right) +D_{2}\left( u_{L}-a_{2}^{-1}u_{L3}\right) =0,\quad x\in \Omega _{L},\ t>0
\label{3.4}
\end{equation}%
\begin{equation}
u_{L}\left( 0,x\right) =u_{L0}\left( x\right) ,\quad \left( t,x\right) \in \left[ 0,T\right] \times \Omega _{L};\quad u_{L}\left\vert \ _{\left( 0,T\right) \times \partial \Omega _{L}}\right. =0.
\label{3.5}
\end{equation}%
by using of same way.

Thus, we derived the problem (\ref{3.3}) - (\ref{3.5}) the study of which
will give we possibility to define properties of solutions $u$ of problem
(1.1%
${{}^1}$%
) - (\ref{3}) on each "cross-section" $\left[ 0,T\right) \times \Omega
_{L}\equiv Q_{L}^{T}$.

In the beginning it is necessary to investigate the existence of the
solution of problem (\ref{3.3}) - (\ref{3.5}) and determine the space where
the existing solutions are contained. Consequently, for study of the
uniqueness of the solution of problem (1.1%
${{}^1}$%
) - (\ref{3}) at first it is necessary to investigate the existence and
uniqueness of the solution for the derived problem (\ref{3.3}) - (\ref{3.5}%
). Therefore we will to investigate of problem (\ref{3.3}) - (\ref{3.5}).

We must to note: For each hyperplane $L\subset R^{3}$ there exists, at
least, one $2$-dimensional subspace of $R^{3}$ that in the given coordinat
system one can determine as $\left( x_{i},x_{j}\right) $ and $%
P_{x_{k}}L=R^{2}$ (e.g. $i,j,k=1,2,3$), i.e. 
\begin{equation*}
L\equiv \left\{ x\in R^{3}\left\vert \ x=\left( x_{i},x_{j},\psi _{L}\left(
x_{i},x_{j}\right) \right) ,\right. \left( x_{i},x_{j}\right) \in
R^{2}\right\}
\end{equation*}%
and 
\begin{equation*}
\Omega \cap L\equiv \left\{ x\in \Omega \left\vert \ x=\left(
x_{i},x_{j},\psi _{L}\left( x_{i},x_{j}\right) \right) ,\right. \left(
x_{i},x_{j}\right) \in P_{x_{k}}\left( \Omega \cap L\right) \right\}
\end{equation*}%
hold, where $\psi _{L}$ is the affine function that is the bijection.

Thereby, in this case functions $u(t,x),\ f(t,x)$ and$\ u_{0}(x)$ can be
represented as 
\begin{equation*}
u(t,x_{i},x_{j},\psi _{L}(x_{i},x_{j}))\equiv v(t,x_{i},x_{j})\text{, }%
f(t,x_{i},x_{j},\psi _{L}(x_{i},x_{j})\equiv \phi (x_{i},x_{j})
\end{equation*}%
and%
\begin{equation*}
u_{0}(x_{i},x_{j},\psi _{L}(x_{i},x_{j}))\equiv v_{0}(x_{i},x_{j})\text{ \ \
on }(0,T)\times P_{x_{k}}\Omega _{L},
\end{equation*}%
respectively.

So, each of these functions can be represented as functions from the
independent variables: $t$, $x_{i}$ and $x_{j}$.

\subsection{\label{SS_I.3.1}\textbf{On Dirichlet to Neumann map}}

As is known (\cite{Nac}, \cite{BehEl}, \cite{DePrZac}, \cite{H-DR}, \cite%
{BelCho} etc.) the Dirichlet to Neumann map is single-value maping if the
homogeneous Dirichlet problem for elliptic equation has only trivial
solution, i.e. zero not is eigenvalue of this problem. Consequently, it is
enough to show that the homogeneous Dirichlet problem for elliptic equation
assosiated to considered problem satisfies of the corresponding conditions
of the results of the mentioned articles. So, we will prove the following

\begin{proposition}
\label{Pr_4.1}The homogeneous Dirichlet problem for elliptic part of problem
(3.3) - (3.5) has only trivial solution.
\end{proposition}

\begin{proof}
If consider the elliptic part of problem (3.3) - (3.5) then we get\ the
problem 
\begin{equation*}
-\Delta u_{L}+Bu_{L}\equiv -\nu \left[ \left( 1+a_{1}^{-2}\right)
D_{1}^{2}+\left( 1+a_{2}^{-2}\right)
D_{2}^{2}+2a_{1}^{-1}a_{2}^{-1}D_{1}D_{2}\right] u_{L}+
\end{equation*}%
\begin{equation*}
\left( u_{L1}-a_{1}^{-1}u_{L3}\right) D_{1}u_{L}+\left(
u_{L2}-a_{2}^{-1}u_{L3}\right) D_{2}u_{L}=0,\ x\in \Omega _{L},\quad
u_{L}\left\vert _{\ \partial \Omega _{L}}\right. =0,
\end{equation*}%
where $\Omega _{L}=\Omega \cap L$.

Let 's show that this problem cannot have nontrivial solutions. This will be
to prove by method of contradiction. Let $u_{L}\in V\left( \Omega
_{L}\right) $ be nontrivial solution of this problem then we get the
following equation 
\begin{equation*}
0=\left\langle -\Delta u_{L}+Bu_{L},u_{L}\right\rangle _{P_{x_{3}}\Omega
_{L}}
\end{equation*}%
hence 
\begin{equation*}
=-\underset{i=1}{\overset{3}{\nu \sum }}\left\langle \left[ \left(
D_{1}^{2}+D_{2}^{2}\right) +\left( a_{1}^{-1}D_{1}+a_{2}^{-1}D_{2}\right)
^{2}\right] u_{Li},u_{Li}\right\rangle _{P_{x_{3}}\Omega _{L}}+
\end{equation*}%
\begin{equation*}
\underset{i=1}{\overset{3}{\sum }}\underset{P_{x_{3}}\Omega _{L}}{\int }%
\left[ u_{L1}D_{1}u_{Li}u_{Li}+u_{L2}D_{2}u_{Li}u_{Li}+\right.
\end{equation*}%
\begin{equation*}
\left. u_{L3}\left( -a_{1}^{-1}D_{1}-a_{2}^{-1}D_{2}\right) u_{Li}u_{Li} 
\right] dx_{1}dx_{2}=
\end{equation*}%
\begin{equation*}
\underset{i=1}{\overset{3}{\nu \sum }}\underset{P_{x_{3}}\Omega _{L}}{\int }%
\left\{ \left( D_{1}u_{Li}\right) ^{2}+\left( D_{2}u_{Li}\right) ^{2}+\left[
\left( a_{1}^{-1}D_{1}+a_{2}^{-1}D_{2}\right) u_{Li}\right] ^{2}\right\}
dx_{1}dx_{2}+
\end{equation*}%
\begin{equation*}
\frac{1}{2}\underset{i=1}{\overset{3}{\sum }}\underset{P_{x_{3}}\Omega _{L}}{%
\int }\left[ u_{L1}D_{1}\left( u_{Li}\right) ^{2}+u_{L2}D_{2}\left(
u_{Li}\right) ^{2}+\right.
\end{equation*}%
\begin{equation*}
\left. u_{L3}\left( -a_{1}^{-1}D_{1}-a_{2}^{-1}D_{2}\right) \left(
u_{Li}\right) ^{2}\right] dx_{1}dx_{2}\geq
\end{equation*}%
\begin{equation*}
\underset{i=1}{\overset{3}{\nu \sum }}\underset{P_{x_{3}}\Omega _{L}}{\int }%
\left[ \left\vert D_{1}u_{Li}\right\vert ^{2}+\left\vert
D_{2}u_{Li}\right\vert ^{2}\right] dx_{1}dx_{2}+
\end{equation*}%
\begin{equation*}
-\frac{1}{2}\underset{i=1}{\overset{3}{\sum }}\underset{P_{x_{3}}\Omega _{L}}%
{\int }\left[ D_{1}u_{L1}+D_{2}u_{L2}+\left(
-a_{1}^{-1}D_{1}-a_{2}^{-1}D_{2}\right) u_{L3}\right] \left\vert
u_{Li}\right\vert ^{2}dx_{1}dx_{2}=
\end{equation*}%
by (\ref{3.4}) 
\begin{equation*}
\underset{i=1}{\overset{3}{\nu \sum }}\underset{P_{x_{3}}\Omega _{L}}{\int }%
\left[ \left\vert D_{1}u_{Li}\right\vert ^{2}+\left\vert
D_{2}u_{Li}\right\vert ^{2}\right] dx_{1}dx_{2}-\frac{1}{2}\underset{i=1}{%
\overset{3}{\sum }}\underset{P_{x_{3}}\Omega _{L}}{\int }\left\vert
u_{Li}\right\vert ^{2}\func{div}u_{L}dx_{1}dx_{2}=
\end{equation*}%
\begin{equation*}
\underset{i=1}{\overset{3}{\nu \sum }}\underset{P_{x_{3}}\Omega _{L}}{\int }%
\left[ \left\vert D_{1}u_{Li}\right\vert ^{2}+\left\vert
D_{2}u_{Li}\right\vert ^{2}\right] dx_{1}dx_{2}>0.
\end{equation*}%
Thus, the obtained contradiction shows that function $u_{L}$ need be zero,
i.e. $u_{L}=0$ holds.

Consequently, the Dirichlet to Neumann map is single-value operator.
\end{proof}

It is well-known that operator $-\Delta :H_{0}^{1}\left( \Omega _{L}\right)
\longrightarrow $ $H^{-1}\left( \Omega _{L}\right) $ generates of the $C_{0}$
semigroup on $H\left( \Omega _{L}\right) $, and since inclusion $%
H_{0}^{1}\left( \Omega _{L}\right) \subset H^{-1}\left( \Omega _{L}\right) $
is compact, therefore $\left( -\Delta \right) ^{-1}$ is the compact operator
in $H^{-1}\left( \Omega _{L}\right) $. Moreover, $-\Delta :H^{1/2}\left(
\partial \Omega _{L}\right) \longrightarrow H^{-1/2}\left( \partial \Omega
_{L}\right) $ and the operator $B:$ $H^{1/2}\left( \partial \Omega
_{L}\right) \longrightarrow H^{-1/2}\left( \partial \Omega _{L}\right) $
also possess appropriate properties of such types.

\section{\label{Sec_I.4}Existence of Solution of Problem (3.3) - (3.5)}

So, assume conditons of Theorem \ref{Th_1} fulfilled, i. e. 
\begin{equation*}
u_{0}\in H^{1/2}\left( \Omega \right) ,\quad f\in L^{2}\left(
0,T;H^{1/2}\left( \Omega \right) \right) ,
\end{equation*}%
then these functions on $\Omega _{L}$, $Q_{L}^{T}$ are correctly defined and
belong to $H\left( \Omega _{L}\right) $, $L^{2}\left( 0,T;H\left( \Omega
_{L}\right) \right) $, respectively. Consequently, we can study problem (\ref%
{3.3}) - (\ref{3.5}) under conditions $u_{0L}\in H\left( \Omega _{L}\right) $
and $f_{L}\in L^{2}\left( 0,T;H\left( \Omega _{L}\right) \right) $, as
independent problem.

By executing according the known argument started by Leray (\cite{Ler1},
see, also \cite{Lio1}, \cite{FoiManRosTem}, \cite{Gal}), the space $V\left(
\Omega _{L}\right) $ of the vector functions $u$ one can determine by same
way as in Definition \ref{D_2.1}: the space $V\left( \Omega _{L}\right) $ is
the closure in $\left( H_{0}^{1}\left( \Omega _{L}\right) \right) ^{3}$\ of 
\begin{equation*}
\left\{ \varphi \left\vert \ \varphi \in \left( C_{0}^{\infty }\left( \Omega
_{L}\right) \right) ^{3},\right. \func{div}\varphi =0\right\}
\end{equation*}%
$\left( W_{0}^{1,2}\left( \Omega _{L}\right) \right) ^{3}$, where $\func{div}
$ is regarded in the sense (\ref{3.4}), in this case the dual space $V\left(
\Omega _{L}\right) $ is determined as $V^{\ast }\left( \Omega _{L}\right) $,
the space $H\left( \Omega _{L}\right) $ also is determined as the closure in 
$\left( L^{2}\left( \Omega _{L}\right) \right) ^{3}$ of 
\begin{equation*}
\left\{ \varphi \left\vert \ \varphi \in \left( C_{0}^{\infty }\left( \Omega
_{L}\right) \right) ^{3},\right. \func{div}\varphi =0\right\}
\end{equation*}%
in the topology of $\left( L^{2}\left( \Omega _{L}\right) \right) ^{3}$.

Consequently, one can determine of space $V\left( Q_{L}^{T}\right) $ as 
\begin{equation*}
V\left( Q_{L}^{T}\right) \equiv L^{2}\left( 0,T;V\left( \Omega _{L}\right)
\right) \cap L^{\infty }\left( 0,T;H\left( \Omega _{L}\right) \right) .
\end{equation*}

Here $\Omega \subset R^{3}$ is bounded domain of $Lip_{loc}$ and $\Omega
_{L}\subset R^{2}$ is subdomain defined in the beginning of Section \ref%
{Sec_I.3} therefore, $\Omega _{L}$ is Lipschitz, $Q_{L}^{T}\equiv \left(
0,T\right) \times \Omega _{L}$.

Let $f_{L}\in L^{2}\left( 0,T;V^{\ast }\left( \Omega _{L}\right) \right) $
and $u_{0L}\in H\left( \Omega _{L}\right) $. Consequently, a solution of
problem (\ref{3.3}) - (\ref{3.5}) will be understood as follows.

So, we can call the solution of this problem: A function $u_{L}\in \mathcal{V%
}\left( Q_{L}^{T}\right) $ is called a solution of the problem (\ref{3.3}) -
(\ref{3.5}) if $u_{L}(t,x^{\prime })$ satisfy the equality 
\begin{equation}
\frac{d}{dt}\left\langle u_{L},v\right\rangle _{\Omega _{L}}-\left\langle
\nu \Delta u_{L},v\right\rangle _{\Omega _{L}}+\left\langle \underset{j=1}{%
\overset{3}{\sum }}u_{Lj}D_{j}u_{L},v\right\rangle _{\Omega
_{L}}=\left\langle f_{L},v\right\rangle _{\Omega _{L}},  \label{3.6}
\end{equation}%
for any $v\in V\left( \Omega _{L}\right) $ and almost everywhere in $\left(
0,T\right) $ and initial condition 
\begin{equation*}
\left\langle u_{L}\left( t\right) ,v\right\rangle \left\vert _{t=0}\right.
=\left\langle u_{0L},v\right\rangle ,
\end{equation*}%
in the sense of $H$, where $\left\langle \circ ,\circ \right\rangle _{\Omega
_{L}}$ is the dual form for the pair of spaces $\left( V\left( \Omega
_{L}\right) ,V^{\ast }\left( \Omega _{L}\right) \right) $ and $\Omega _{L}$
is Lipschitz. Where $x^{\prime }\in \Omega _{L}$ is $x^{\prime }\equiv
\left( x_{1},x_{2}\right) $ (according to our selection of the $L$) and $%
\mathcal{V}\left( Q_{L}^{T}\right) $ is 
\begin{equation*}
\mathcal{V}\left( Q_{L}^{T}\right) \equiv \left\{ w\left\vert \ w\in V\left(
Q_{L}^{T}\right) ,\ w^{\prime }\in L^{\frac{4}{3}}\left( 0,T;V^{\ast }\left(
\Omega _{L}\right) \right) \right. \right\} .
\end{equation*}

We will lead of the proof of this problem in five-steps as indepandent
problem.

\subsection{\label{SS_I.4.1}\textbf{A priori estamates}}

In order to derive of the a priori estimates for the possible solutions of
the problem we will apply of the usual approach. By substituting in (\ref%
{3.6}) of the function $u_{L}$ instead of the function $v$, we get 
\begin{equation}
\frac{d}{dt}\left\langle u_{L},u_{L}\right\rangle _{\Omega
_{L}}-\left\langle \nu \Delta u_{L},u_{L}\right\rangle _{\Omega
_{L}}+\left\langle \underset{j=1}{\overset{3}{\sum }}u_{Lj}D_{j}u_{L},u_{L}%
\right\rangle _{\Omega _{L}}=\left\langle f_{L},u_{L}\right\rangle _{\Omega
_{L}}.  \label{3.6'}
\end{equation}%
Thence, by making the known calculations, taking into account of the
condition on $\Omega _{L}$ and (\ref{3.4}), and also of calculations (\ref%
{3.1}) that carried out in the previous Section, we derive 
\begin{equation*}
\frac{1}{2}\frac{d}{dt}\left\Vert u_{L}\right\Vert _{H\left( \Omega
_{L}\right) }^{2}\left( t\right) +\nu \left( 1+a_{1}^{-2}\right) \left\Vert
D_{1}u_{L}\right\Vert _{H\left( \Omega _{L}\right) }^{2}\left( t\right) +
\end{equation*}%
\begin{equation}
\nu \left( 1+a_{2}^{-2}\right) \left\Vert D_{2}u_{L}\right\Vert _{H\left(
\Omega _{L}\right) }^{2}\left( t\right) +2\nu
a_{1}^{-1}a_{2}^{-1}\left\langle D_{1}u_{L},D_{2}u_{L}\right\rangle _{\Omega
_{L}}\left( t\right) =\left\langle f_{L},u_{L}\right\rangle _{\Omega _{L}},
\label{3.7}
\end{equation}%
where $\left\langle g,h\right\rangle _{\Omega _{L}}=\underset{i=1}{\overset{3%
}{\sum }}\underset{P_{x_{3}}\Omega _{L}}{\int }g_{i}h_{i}dx_{1}dx_{2}$ for
any $g,h\in H\left( \Omega _{L}\right) $, or $g\in \left( W^{1,2}\left(
\Omega _{L}\right) \right) ^{3}$ and $h\in \left( W^{-1,2}\left( \Omega
_{L}\right) \right) ^{3}$, respectively. We will show the correctness of (%
\ref{3.7}), and to this end we shall prove the correctness of each term of
this sum, separately.

So, using of (\ref{3.6'}) we get 
\begin{equation*}
-\nu \left\langle \Delta u_{L}\left( t\right) ,u_{L}\left( t\right)
\right\rangle _{\Omega _{L}}=
\end{equation*}%
\begin{equation*}
-\underset{i=1}{\overset{3}{\nu \sum }}\left\langle \left[ \left(
1+a_{1}^{-2}\right) D_{1}^{2}+\left( 1+a_{2}^{-2}\right)
D_{2}^{2}+2a_{1}^{-1}a_{2}^{-1}D_{1}D_{2}\right] u_{Li},u_{Li}\right\rangle
_{P_{x_{3}}\Omega _{L}}=
\end{equation*}%
\begin{equation*}
\underset{i=1}{\overset{3}{\nu \sum }}\underset{P_{x_{3}}\Omega _{L}}{\int }%
\left[ \left( 1+a_{1}^{-2}\right) \left( D_{1}u_{Li}\right) ^{2}+\left(
1+a_{2}^{-2}\right) \left( D_{2}u_{Li}\right) ^{2}+\right.
\end{equation*}%
\begin{equation*}
\left. 2a_{1}^{-1}a_{2}^{-1}D_{1}u_{Li}D_{2}u_{Li}\right] dx_{1}dx_{2}
\end{equation*}%
thus is obtained the sum reducible in (\ref{3.7}).

Whence isn`t difficult to seen, that if to estimate of the last adding in
the above mentioned sum then one will received 
\begin{equation*}
-\nu \left\langle \Delta u_{L}\left( t\right) ,u_{L}\left( t\right)
\right\rangle _{\Omega _{L}}\geq
\end{equation*}%
\begin{equation}
\nu \left[ \left\Vert D_{1}u_{L}\right\Vert _{H\left( \Omega _{L}\right)
}^{2}\left( t\right) +\left\Vert D_{2}u_{L}\right\Vert _{H\left( \Omega
_{L}\right) }^{2}\left( t\right) \right] .  \label{3.8}
\end{equation}

Now consider the trilinear form from (\ref{3.6'}) 
\begin{equation*}
\left\langle \underset{j=1}{\overset{3}{\sum }}u_{Lj}D_{j}u_{L},u_{L}\right%
\rangle _{\Omega _{L}}=
\end{equation*}%
due to (\ref{3.3}) we get 
\begin{equation*}
\underset{i=1}{\overset{3}{\sum }}\underset{P_{x_{3}}\Omega _{L}}{\int }%
\left[ u_{L1}D_{1}u_{Li}u_{Li}+u_{L2}D_{2}u_{Li}u_{Li}+\right.
\end{equation*}%
\begin{equation*}
\left. u_{L3}\left( -a_{1}^{-1}D_{1}-a_{2}^{-1}D_{2}\right) u_{Li}u_{Li} 
\right] dx_{1}dx_{2}=
\end{equation*}%
\begin{equation*}
\frac{1}{2}\underset{i=1}{\overset{3}{\sum }}\underset{P_{x_{3}}\Omega _{L}}{%
\int }\left[ u_{L1}D_{1}\left( u_{Li}\right) ^{2}+u_{L2}D_{2}\left(
u_{Li}\right) ^{2}+\right.
\end{equation*}%
\begin{equation*}
\left. u_{L3}\left( -a_{1}^{-1}D_{1}-a_{2}^{-1}D_{2}\right) \left(
u_{Li}\right) ^{2}\right] dx_{1}dx_{2}=
\end{equation*}%
\begin{equation*}
-\frac{1}{2}\underset{i=1}{\overset{3}{\sum }}\underset{P_{x_{3}}\Omega _{L}}%
{\int }\left[ D_{1}u_{L1}+D_{2}u_{L2}+\left(
-a_{1}^{-1}D_{1}-a_{2}^{-1}D_{2}\right) u_{L3}\right] \left( u_{Li}\right)
^{2}dx_{1}dx_{2}=
\end{equation*}%
hence by (\ref{3.4}) 
\begin{equation}
-\frac{1}{2}\underset{i=1}{\overset{3}{\sum }}\underset{P_{x_{3}}\Omega _{L}}%
{\int }\left( u_{Li}\right) ^{2}\func{div}u_{L}dx_{1}dx_{2}=0.  \label{3.9}
\end{equation}

Consequently, the correctness of equation (\ref{3.7}) is proved.

From (\ref{3.7}) in view of (\ref{3.8})-(\ref{3.9}) is derived the following
inequality 
\begin{equation*}
\frac{1}{2}\frac{d}{dt}\left\Vert u_{L}\right\Vert _{H\left( \Omega
_{L}\right) }^{2}\left( t\right) +
\end{equation*}%
\begin{equation}
\nu \underset{i=1}{\overset{3}{\sum }}\underset{P_{x_{3}}\Omega _{L}}{\int }%
\left[ \left( D_{1}u_{Li}\right) ^{2}+\left( D_{2}u_{Li}\right) ^{2}\right]
dx_{1}dx_{2}\leq \underset{P_{x_{3}}\Omega _{L}}{\int }\left\vert \left(
f_{L}\cdot u_{L}\right) \right\vert dx_{1}dx_{2}  \label{3.10}
\end{equation}%
which give we the following a priori estimates 
\begin{equation}
\left\Vert u_{L}\right\Vert _{H\left( \Omega _{L}\right) }\left( t\right)
\leq C\left( f_{L},u_{L0},mes\Omega \right) ,  \label{3.11}
\end{equation}%
\begin{equation}
\left\Vert D_{1}u_{L}\right\Vert _{H\left( \Omega _{L}\right) }+\left\Vert
D_{2}u_{L}\right\Vert _{H\left( \Omega _{L}\right) }\leq C\left(
f_{L},u_{L0},mes\Omega \right) ,  \label{3.12}
\end{equation}%
where $C\left( f_{L},u_{L0},mes\Omega \right) >0$ is the constant that is
independent of $u_{L}$. Consequently, any possible solution of this problem
belong to a bounded subset of the space $V\left( Q_{L}^{T}\right) $.

Thus, it is remain to receive of the necessary a priori estimate for $\frac{%
\partial u_{L}}{\partial t}$ and to study of properties of the thrilinear
term in order to prove of the existence theorem. \footnote{%
It should be noted that if the represantation of $%
\Omega
_{L}$ by coordinates $(x_{1},x_{2})$ not is best for the definition of the
appropriate integral, then we will select other coordinates: either $%
(x_{1},x_{3})$ or $(x_{2},x_{3})$ instead of $(x_{1},x_{2})$, which is best
for our goal (that must exist by virtue of the definition of $%
\Omega
$).
\par
{}}

\subsection{\label{SS_I.4.2}Boundedness of the trilinear form}

Boundedness of the trilinear form $b_{L}\left( u_{L},u_{L},v\right) $ from (%
\ref{3.6}) follows from the next result.

\begin{proposition}
\label{P_3.1}Let $u_{L}\in V\left( Q_{L}^{T}\right) \cap L^{\infty }\left(
0,T;H\right) $, $v\in V\left( \Omega _{L}\right) $ and $B$ is the operator
defined by 
\begin{equation*}
\left\langle B\left( u_{L}\right) ,v\right\rangle _{\Omega _{L}}=b_{L}\left(
u_{L},u_{L},v\right) =\left\langle \underset{j=1}{\overset{3}{\sum }}%
u_{Lj}D_{j}u_{L},v\right\rangle _{\Omega _{L}}
\end{equation*}%
then $B\left( u_{L}\right) $ belongs to bounded subset of $L^{\frac{3}{2}%
}\left( 0,T;V^{\ast }\left( \Omega _{L}\right) \right) $.
\end{proposition}

\begin{proof}
At first we will show boundedness of the operator $B$ acting from $V\left(
\Omega _{L}\right) \times V\left( \Omega _{L}\right) $ to $V^{\ast }\left(
\Omega _{L}\right) $ for a. e. $t\in \left( 0,T\right) $. We have 
\begin{equation*}
\left\langle B\left( u_{L}\right) ,v\right\rangle _{\Omega
_{L}}=\left\langle \underset{j=1}{\overset{3}{\sum }}u_{Lj}D_{j}u_{L},v%
\right\rangle _{\Omega _{L}}=
\end{equation*}%
\begin{equation*}
\underset{i=1}{\overset{3}{\sum }}\underset{P_{x_{3}}\Omega _{L}}{\int }%
\left[ u_{L1}D_{1}u_{Li}v_{i}+u_{L2}D_{2}u_{Li}v_{i}+u_{L3}\left(
-a_{1}^{-1}D_{1}-a_{2}^{-1}D_{2}\right) u_{Li}v_{i}\right] dx_{1}dx_{2}=
\end{equation*}%
\begin{equation*}
\underset{i=1}{\overset{3}{\sum }}\underset{P_{x_{3}}\Omega _{L}}{\int }%
\left[ \left( u_{L1}-a_{1}^{-1}u_{L3}\right) D_{1}u_{Li}v_{i}+\left(
u_{L2}-a_{2}^{-1}u_{L3}\right) D_{2}u_{Li}v_{i}\right] dx_{1}dx_{2}=
\end{equation*}%
\begin{equation}
\underset{i=1}{\overset{3}{\sum }}\underset{P_{x_{3}}\Omega _{L}}{\int }%
\left[ \left( u_{L1}-a_{1}^{-1}u_{L3}\right) D_{1}+\left(
u_{L2}-a_{2}^{-1}u_{L3}\right) D_{2}\right] u_{Li}v_{i}dx_{1}dx_{2}
\label{3.13}
\end{equation}%
due of (\ref{3.4}) and of the definition \ref{D_3.1}.

Hence follows 
\begin{equation*}
\left\vert \left\langle B\left( u_{L}\right) ,v\right\rangle \right\vert
\leq \underset{i=1}{\overset{3}{\sum }}\underset{P_{x_{3}}\Omega _{L}}{\int }%
c\left( \left\vert u_{L1}\right\vert +\left\vert u_{L2}\right\vert
+\left\vert u_{L3}\right\vert \right) \left( \left\vert
D_{1}u_{Li}\right\vert +\left\vert D_{2}u_{Li}\right\vert \right)
v_{i}dx_{1}dx_{2}\leq
\end{equation*}%
\begin{equation}
c\left\Vert u_{L}\right\Vert _{L^{4}\left( \Omega _{L}\right) }\left\Vert
u_{L}\right\Vert _{V}\left\Vert v\right\Vert _{L^{4}\left( \Omega
_{L}\right) }\Longrightarrow \left\Vert B\left( u_{L}\right) \right\Vert
_{V^{\ast }}\leq c\left\Vert u_{L}\right\Vert _{L^{4}\left( \Omega
_{L}\right) }\left\Vert u_{L}\right\Vert _{V}  \label{3.14}
\end{equation}%
due of $V\left( \Omega _{L}\right) \subset L^{4}\left( \Omega _{L}\right) $.
This also shows that operator $B:V\left( \Omega _{L}\right) \longrightarrow
V^{\ast }\left( \Omega _{L}\right) $ is bounded, and continuous for a. e. $%
t>0$.

Finally, we obtain needed result using above inequality and the well-known
inequality, which is valid in two-dimension space (see, \cite{Lad1}, \cite%
{Lio1}, \cite{Tem1}) 
\begin{equation*}
\overset{T}{\underset{0}{\int }}\left\Vert B\left( u_{L}\left( t\right)
\right) \right\Vert _{V^{\ast }}^{\frac{4}{3}}dt\leq c\overset{T}{\underset{0%
}{\int }}\left( \left\Vert u_{L}\left( t\right) \right\Vert
_{L^{4}}\left\Vert u_{L}\right\Vert _{V}\right) ^{\frac{4}{3}}dt\leq
\end{equation*}%
according to Gagliardo--Nirenberg inequality we get 
\begin{equation*}
c_{1}\overset{T}{\underset{0}{\int }}\left\Vert u_{L}\left( t\right)
\right\Vert _{L^{2}}^{\frac{2}{3}}\left\Vert u_{L}\right\Vert _{V}^{2}dt\leq
c_{1}\left\Vert u_{L}\right\Vert _{L^{\infty }\left( 0,T;H\right) }^{\frac{2%
}{3}}\overset{T}{\underset{0}{\int }}\left\Vert u_{L}\right\Vert
_{V}^{2}dt\Longrightarrow
\end{equation*}%
\begin{equation}
\left\Vert B\left( u_{L}\right) \right\Vert _{L^{\frac{4}{3}}\left(
0,T;V^{\ast }\right) }\leq c_{1}\left\Vert u_{L}\right\Vert _{L^{\infty
}\left( 0,T;H\right) }^{\frac{1}{2}}\left\Vert u_{L}\right\Vert
_{L^{2}\left( 0,T;V\right) }^{\frac{3}{2}}.  \label{3.15}
\end{equation}
\end{proof}

Moreover, is proved that 
\begin{equation*}
B:L^{2}\left( 0,T;V\left( \Omega _{L}\right) \right) \cap L^{\infty }\left(
0,T;H\right) =V\left( Q_{L}^{T}\right) \longrightarrow L^{\frac{4}{3}}\left(
0,T;V^{\ast }\right)
\end{equation*}%
is bounded operator.

\subsection{\label{SS_I.4.3}Boundedness of $u^{\prime }$}

Sketch of the proof that $u^{\prime }$ belongs to bounded subset of the
space $L^{\frac{4}{3}}\left( 0,T;V^{\ast }\left( \Omega _{L}\right) \right) $%
. It is possible to draw the following conclusion based due received of a
priori estimates, proposition \ref{P_3.1} and reflexivity of all used
spaces: If we will use of the Faedo-Galerkin's method for investigation then
for the approximate solutions we obtain estimates of such type as (\ref{3.11}%
), (\ref{3.12}) and (\ref{3.15}). Indeed since $V\left( \Omega _{L}\right) $
is a separable there exists a countable subset of linearly independent
elements $\left\{ w_{i}\right\} _{i=1}^{\infty }\subset V\left( \Omega
_{L}\right) $, which is total in $V\left( \Omega _{L}\right) $. For each $m$
we can define an approximate solution of $u_{Lm}$ (\ref{3.6}) as follows

\begin{equation}
u_{Lm}=\overset{m}{\underset{i=1}{\sum }}u_{Lm}^{i}\left( t\right)
w_{i},\quad m=1,\ 2,....,  \label{5.4}
\end{equation}%
where $u_{Lm}^{i}\left( t\right) $, $i=\overline{1,m}$ be unknown functions
that will be determined as solutions of the following system of the
differential equations that is received according to equation (\ref{3.6}) 
\begin{equation*}
\left\langle \frac{d}{dt}u_{Lm},w_{j}\right\rangle _{\Omega
_{L}}=\left\langle \nu \Delta u_{Lm},w_{j}\right\rangle _{\Omega
_{L}}+b_{L}\left( u_{Lm},u_{Lm},w_{j}\right) +
\end{equation*}%
\begin{equation}
+\left\langle f_{L},w_{j}\right\rangle _{\Omega _{L}},\quad t\in \left( 0,T 
\right] ,\quad j=\overline{1,m},\quad  \label{5.5}
\end{equation}%
\begin{equation*}
u_{Lm}\left( 0\right) =u_{0Lm}.
\end{equation*}%
Here we assume $\left\{ u_{0Lm}\right\} _{m=1}^{\infty }\subset H\left(
\Omega _{L}\right) $ be such sequence that $u_{0Lm}\longrightarrow u_{0L}$
in $H\left( \Omega _{L}\right) $ as $m\longrightarrow \infty $. (Since $%
V\left( \Omega _{L}\right) $ is everywhere dense in $H\left( \Omega
_{L}\right) $ one can determine $u_{0Lm}$ by using the total system $\left\{
w_{i}\right\} _{i=1}^{\infty }$). \ 

So, with use (\ref{5.4}) in (\ref{5.5}) we obtain 
\begin{equation*}
\overset{m}{\underset{j=1}{\sum }}\left\langle w_{j},w_{i}\right\rangle
_{\Omega _{L}}\frac{d}{dt}u_{Lm}^{j}\left( t\right) -\nu \overset{m}{%
\underset{j=1}{\sum }}\left\langle \Delta w_{j},w_{i}\right\rangle _{\Omega
_{L}}u_{Lm}^{j}\left( t\right) +
\end{equation*}%
\begin{equation*}
\overset{m}{\underset{j,k=1}{\sum }}b_{L}\left( w_{j},w_{k},w_{i}\right)
u_{Lm}^{j}\left( t\right) u_{Lm}^{k}\left( t\right) =\left\langle
f_{L}\left( t\right) ,w_{i}\right\rangle _{\Omega _{L}},\ i=\overline{1,m}.
\end{equation*}%
As the matrix generated by $\left\langle w_{i},w_{j}\right\rangle _{\Omega
_{L}}$, $i,j=\overline{1,m}$ is nonsingular then its inverse exists. Thanks
this from the previous equations we will derive the following Cauchy problem
for the system of the nonlinear ordinary differential equations for unknown
functions $u_{Lm}^{i}\left( t\right) $, $i=1,...,m$. 
\begin{equation*}
\frac{du_{Lm}^{i}\left( t\right) }{dt}=\overset{m}{\underset{j=1}{\sum }}%
c_{i,j}\left\langle f_{L}\left( t\right) ,w_{j}\right\rangle _{\Omega
_{L}}-\nu \overset{m}{\underset{j=1}{\sum }}d_{i,j}u_{Lm}^{j}\left( t\right)
+
\end{equation*}%
\begin{equation}
\overset{m}{\underset{j,k=1}{\sum }}h_{ijk}u_{Lm}^{j}\left( t\right)
u_{Lm}^{k}\left( t\right) ,  \label{5.6}
\end{equation}%
\begin{equation*}
u_{Lm}^{i}\left( 0\right) =u_{0Lm}^{i},\quad i=1,...,m,\quad m=1,\ 2,...
\end{equation*}%
where $u_{0Lm}^{i}$ is $i^{th}$ component of $u_{0L}$ in representation $%
u_{0L}=\overset{\infty }{\underset{k=1}{\sum }}u_{0Lm}^{k}w_{k}$.

The Cauchy problem for the system of the nonlinear ordinary differential
equations (\ref{5.6}) has solution, which defined on whole of interval $%
(0,T] $ due of uniformity of estimations received in subsections \ref%
{SS_I.4.1} and \ref{SS_I.4.2}. Consequently, the approximate solutions $%
u_{Lm}$ exist and belong to a bounded subset of $W^{1,\frac{4}{3}}\left(
0,T;V^{\ast }\left( \Omega _{L}\right) \right) $ for every $m=1,\ 2,...$
since the right side of (\ref{5.6}) belong to a bounded subset of $%
L^{2}\left( 0,T;V^{\ast }\left( \Omega _{L}\right) \right) $ as were proved
in subsections \ref{SS_I.4.1} and \ref{SS_I.4.2}, and also by virtue of the
next lemma.

\begin{lemma}
(\cite{Tem1}) Let $X$ be a given Banach space with dual $X^{\ast }$ and let $%
u$ and $g$ be two functions belonging to $L^{1}\left( a,b;X\right) $. Then,
the following three conditions are equivalent

\textit{(i)} $u$ is a. e. equal to a primitive function of $g$, 
\begin{equation*}
u\left( t\right) =\xi +\underset{a}{\overset{t}{\int }}g\left( s\right)
ds,\quad \xi \in X,\quad \text{a.e. }t\in \left[ a,b\right]
\end{equation*}%
\textit{(ii)} For each test function $\varphi \in D\left( \left( a,b\right)
\right) $, 
\begin{equation*}
\underset{a}{\overset{b}{\int }}u\left( t\right) \varphi ^{\prime }\left(
t\right) dt=-\underset{a}{\overset{b}{\int }}g\left( t\right) \varphi \left(
t\right) dt,\quad \varphi ^{\prime }=\frac{d\varphi }{dt}
\end{equation*}%
(iii) For each $\eta \in X^{\ast }$, 
\begin{equation*}
\frac{d}{dt}\left\langle u,\eta \right\rangle =\left\langle g,\eta
\right\rangle
\end{equation*}%
in the scalar distribution sense, on $(a,b)$. If \textit{(i) - (iii)} are
satisfied $u$, in particular, is a. e. equal to a continuous function from $%
[a,b]$ into $X$.
\end{lemma}

It isn%
\'{}%
t difficult to see that if take $\forall v\in V\left( \Omega _{L}\right) $
instead of $w_{k}$ and pass to limit according to $m\longrightarrow \infty $
in equation (\ref{5.5}) (may be by subsequence $\left\{ u_{Lm_{\mathit{l}%
}}\right\} _{\mathit{l}=1}^{\infty }$ of this sequence, is known that such
subsequence exists) then we get \ 
\begin{equation}
\left\langle \frac{d}{dt}u_{L},v\right\rangle _{\Omega _{L}}=\left\langle
f_{L}+\nu \Delta u_{L}-\chi ,v\right\rangle _{\Omega _{L}},  \label{5.6.1}
\end{equation}%
due of fullness of the class $\left\{ w_{i}\right\} _{i=1}^{\infty }$ in $%
V\left( \Omega _{L}\right) .$ Where function $\chi $ belongs to $L^{\frac{4}{%
3}}\left( 0,T;V^{\ast }\left( \Omega _{L}\right) \right) $ and is determined
by equality 
\begin{equation*}
\underset{\mathit{l}\longrightarrow \infty }{\lim }\left\langle B\left(
u_{Lm_{\mathit{l}}}\right) ,v\right\rangle _{\Omega _{L}}=\left\langle \chi
,v\right\rangle _{\Omega _{L}}
\end{equation*}%
that shown in the above section. So, we obtain that in (\ref{5.6.1}) the
right side belong to $L^{\frac{4}{3}}\left( 0,T\right) $ then the left side
also belongs to $L^{2}\left( 0,T\right) $ according to above a priori
estimates and Proposition \ref{P_3.1}, i. e. 
\begin{equation*}
\frac{du_{L}}{dt}\in L^{\frac{4}{3}}\left( 0,T;V^{\ast }\left( \Omega
_{L}\right) \right) .
\end{equation*}%
Consequently, the following result is proven.

\begin{proposition}
\label{P_4.1}Under above mentioned conditions $u_{L}^{\prime }$ belongs to a
bounded subset of the space $L^{\frac{4}{3}}\left( 0,T;V^{\ast }\left(
\Omega _{L}\right) \right) $.
\end{proposition}

From above results of this section by virtue of the abstract form of
Riesz-Fischer theorem follows

\begin{corollary}
\label{C_4.1}Under above mentioned conditions function $u_{L}$ belongs to a
bounded subset of the space $\mathcal{V}\left( Q_{L}^{T}\right) $, where 
\begin{equation}
\mathcal{V}\left( Q_{L}^{T}\right) \equiv V\left( Q_{L}^{T}\right) \cap W^{1,%
\frac{4}{3}}\left( 0,T;V^{\ast }\left( \Omega _{L}\right) \right) .
\label{3.16}
\end{equation}
\end{corollary}

Thus for the proof that $u_{L}$ is the solution of poblem (\ref{3.3}) - (\ref%
{3.5}) or (\ref{3.6}) remains to show that $\chi =B\left( u_{L}\right) $ or $%
\left\langle \chi ,v\right\rangle _{\Omega _{L}}=b_{L}\left(
u_{L},u_{L},v\right) $ for $\forall v\in V\left( \Omega _{L}\right) $. \ 

\subsection{\label{SS_I.4.4}Weakly compactness of operator $B$}

\begin{proposition}
\label{P_3.2}Operator $B:V\left( Q_{L}^{T}\right) \longrightarrow L^{\frac{4%
}{3}}\left( 0,T;V^{\ast }\left( \Omega _{L}\right) \right) $ is weakly
compact operator, i.e. any weakly convergent sequence $\left\{
u_{L}^{m}\right\} _{1}^{\infty }\subset V\left( Q_{L}^{T}\right) $ posses
such subsequence $\left\{ u_{L}^{m_{k}}\right\} _{1}^{\infty }\subset
\left\{ u_{L}^{m}\right\} _{1}^{\infty }$, that $\left\{ B\left(
u_{L}^{m_{k}}\right) \right\} _{1}^{\infty }$ weakly converged in $L^{\frac{4%
}{3}}\left( 0,T;V^{\ast }\left( \Omega _{L}\right) \right) $.
\end{proposition}

\begin{proof}
Let sequence $\left\{ u_{L}^{m}\right\} _{1}^{\infty }\subset V\left(
Q_{L}^{T}\right) $ be weakly converge to $u_{L}^{0}$ in $V\left(
Q_{L}^{T}\right) $. Then there exists such subsequence $\left\{
u_{L}^{m_{k}}\right\} _{1}^{\infty }\subset \left\{ u_{L}^{m}\right\}
_{1}^{\infty }$ that $u_{L}^{m_{k}}\longrightarrow u_{L}^{0}$ in $%
L^{2}\left( 0,T;H\right) $, due of the known theorems on the compactness of
the embedding, particullary, as known the following embedding 
\begin{equation*}
\mathcal{V}\left( Q_{L}^{T}\right) \equiv L^{2}\left( 0,T;V\left( \Omega
_{L}\right) \right) \cap W^{1,\frac{4}{3}}\left( 0,T;V^{\ast }\left( \Omega
_{L}\right) \right) \subset L^{2}\left( 0,T;H\right)
\end{equation*}%
is compact (see, e. g. \cite{Lio1}, \cite{Tem1}, \cite{Sol3}).

Actually it is enough to show that the operator defined by expression $%
\underset{j=1}{\overset{3}{\sum }}u_{Lj}D_{j}u_{L}$ is weakly compact from $%
\mathcal{V}\left( Q_{L}^{T}\right) $ to $L^{\frac{4}{3}}\left( 0,T;V^{\ast
}\left( \Omega _{L}\right) \right) $. From a priori estimations and
Proposition \ref{P_3.1} follow that operator $B:\mathcal{V}\left(
Q_{L}^{T}\right) \longrightarrow L^{\frac{4}{3}}\left( 0,T;V^{\ast }\left(
\Omega _{L}\right) \right) $ is bounded, i.e. the image of operator $B$ of
each bounded subset of space $\mathcal{V}\left( Q_{L}^{T}\right) $ is the
bounded subset of space $L^{\frac{4}{3}}\left( 0,T;V^{\ast }\left( \Omega
_{L}\right) \right) $.\ 

From above compactness theorem follows the sequence $\left\{
u_{L}^{m}\right\} _{1}^{\infty }$ posses some subsequence $\left\{
u_{L}^{m_{k}}\right\} _{1}^{\infty }\subset \left\{ u_{L}^{m}\right\}
_{1}^{\infty }$ strongly convergent to some element $u_{L}$ of $L^{2}\left(
0,T;H\right) $ in the space $L^{2}\left( 0,T;H\right) $. Consequently, $%
B\left( \left\{ u_{L}^{m_{k}}\right\} _{1}^{\infty }\right) $ belongs of
bounded subset of space $L^{\frac{4}{3}}\left( 0,T;V^{\ast }\left( \Omega
_{L}\right) \right) $. Thence lead that there is such element $\chi \in L^{%
\frac{4}{3}}\left( 0,T;V^{\ast }\left( \Omega _{L}\right) \right) $ that
sequence $B\left( u_{L}^{m_{k}}\right) $ weakly converges to $\chi $ when $%
m_{k}\nearrow \infty $, i.e. 
\begin{equation}
B\left( u_{L}^{m_{k}}\right) \rightharpoonup \chi \quad \text{ in }L^{\frac{4%
}{3}}\left( 0,T;V^{\ast }\left( \Omega _{L}\right) \right)  \label{3.17}
\end{equation}%
due of the reflexivity of this space (there exists, at least, such
subsequence that this occurs).

If we set the vector space 
\begin{equation*}
\mathcal{C}^{1}\left( \overline{Q}_{L}\right) \equiv \left\{ v\left\vert \
v_{i}\in C^{1}\left( \left[ 0,T\right] ;C_{0}^{1}\left( \overline{\Omega _{L}%
}\right) \right) ,\right. i=1,2,3\right\}
\end{equation*}%
and consider the trilinear form 
\begin{equation*}
\underset{0}{\overset{T}{\int }}\left\langle B\left( u_{L}^{m}\right)
,v\right\rangle _{\Omega _{L}}dt=\underset{0}{\overset{T}{\int }}b\left(
u_{L}^{m},u_{L}^{m},v\right) dt=\underset{0}{\overset{T}{\int }}\left\langle 
\underset{j=1}{\overset{3}{\sum }}u_{Lj}^{m}D_{j}u_{L}^{m},v\right\rangle
_{\Omega _{L}}dt=
\end{equation*}%
for $v\in \mathcal{C}^{1}\left( \overline{Q}_{L}\right) $, then we get 
\begin{equation}
-\underset{i=1}{\overset{3}{\sum }}\underset{0}{\overset{T}{\int }}\underset{%
P_{x_{3}}\Omega _{L}}{\int }\left[ \left(
u_{Li}^{m}u_{L1}^{m}-a_{1}^{-1}u_{Li}^{m}u_{L3}^{m}\right) D_{1}v_{i}+\left(
u_{Li}^{m}u_{L2}^{m}-a_{2}^{-1}u_{Li}^{m}u_{L3}^{m}\right) D_{2}v_{i}\right]
dx_{1}dx_{2}dt.  \label{3.17a}
\end{equation}

according to (\ref{3.13}). Now if we take arbitrary term in this sum
separately then it isn't difficult to see that the following convergences
are true, because $u_{Li}^{m_{k}}\longrightarrow u_{Li}$ in $L^{2}\left(
0,T;H\right) $ and $u_{Li}^{m_{k}}\rightharpoonup u_{Li}$ in $L^{\infty
}\left( 0,T;H\right) $ $\ast -$ weakly\ since $u_{L}^{m}$ belong to a
bounded subset of $\mathcal{V}\left( Q_{L}^{T}\right) $ and (\ref{3.17a}) is
fulfill for each term.

Thus passing to the limit when $m_{k}\nearrow \infty $ we obtain 
\begin{equation*}
\chi =B\left( u_{L}\right) \Longrightarrow B\left( u_{L}^{m_{k}}\right)
\rightharpoonup B\left( u_{L}\right) \quad \text{ in the distribution sense}.
\end{equation*}
Whence using the density of $\mathcal{C}^{1}\left( \overline{Q}_{L}\right) $
in $\mathcal{V}\left( Q_{L}^{T}\right) $, and as $B\left(
u_{L}^{m_{k}}\right) \rightharpoonup \chi $ takes place in the space $L^{%
\frac{4}{3}}\left( 0,T;V^{\ast }\left( \Omega _{L}\right) \right) $ we get
that $\chi =B\left( u_{L}\right) $ also takes place in this space.
\end{proof}

Consequently, we proved the existence of the function $u_{L}\in \mathcal{V}%
\left( Q_{L}^{T}\right) $ that satisfies equation (\ref{3.6}) by applying to
this problem of the Faedo-Galerkin method and using the above mentioned
results.

\subsection{\label{SS_I.4.5}Realisation of the initial condition}

We will lead the proof of the realisation of initial condition according to
same way as in \cite{Tem1} (see, also \cite{Lad1}, \cite{Lio1}).

Let $\phi $ be a continuously differentiable function on $[0,T]$ with $\phi
(T)=0$. With multiplying (\ref{5.5}) by $\phi (t)$, and then the first term
integrating by parts we leads to equation 
\begin{equation*}
-\underset{0}{\overset{T}{\int }}\left\langle u_{Lm},\frac{d}{dt}\phi
(t)w_{j}\right\rangle _{\Omega _{L}}dt=\underset{0}{\overset{T}{\int }}%
\left\langle \nu \Delta u_{Lm},\phi (t)w_{j}\right\rangle _{\Omega _{L}}dt+
\end{equation*}%
\begin{equation*}
\underset{0}{\overset{T}{\int }}b\left( u_{Lm},u_{Lm},\phi (t)w_{j}\right)
dt+\underset{0}{\overset{T}{\int }}\left\langle f_{L},\phi
(t)w_{j}\right\rangle _{\Omega _{L}}dt+\left\langle u_{0Lm},\phi
(0)w_{j}\right\rangle _{\Omega _{L}}.
\end{equation*}

One can pass to the limit with respect to subsequence $\left\{
u_{Lm_{l}}\right\} _{l=1}^{\infty }$ of the sequence $\left\{ u_{Lm}\right\}
_{m=1}^{\infty }$ in the equality mentioned above owing to the results
proved in the previous subsections. Then we find the equation 
\begin{equation*}
-\underset{0}{\overset{T}{\int }}\left\langle u_{L},\frac{d}{dt}\phi
(t)w_{j}\right\rangle _{\Omega _{L}}dt=\underset{0}{\overset{T}{\int }}%
\left\langle \nu \Delta u_{L},\phi (t)w_{j}\right\rangle _{\Omega _{L}}dt+
\end{equation*}%
\begin{equation}
\underset{0}{\overset{T}{\int }}b\left( u_{L},u_{L},\phi (t)w_{j}\right) dt+%
\underset{0}{\overset{T}{\int }}\left\langle f_{L},\phi
(t)w_{j}\right\rangle _{\Omega _{L}}dt+\left\langle u_{0L},\phi
(0)w_{j}\right\rangle _{\Omega _{L}},  \label{5.8}
\end{equation}%
that holds for each $w_{j}$, $j=1,2,...$. Consequently, this equality holds
for any finite linear combination of the $w_{j}$ and moreover due of
continuity (\ref{5.8}) remains true and for any $v\in V\left( \Omega
_{L}\right) $.

Whence, one can draw conclusion that function $u_{L}$ satisfies equation (%
\ref{3.6}) in the distribution sense.

Now if multiply (\ref{3.6}) by $\phi (t)$, and integrate with respect to $t$
after integrating the first term by parts, then we get 
\begin{equation*}
-\underset{0}{\overset{T}{\int }}\left\langle u_{L},v\frac{d}{dt}\phi
(t)\right\rangle _{\Omega _{L}}dt-\underset{0}{\overset{T}{\int }}%
\left\langle \nu \Delta u_{L},\phi (t)v\right\rangle _{\Omega _{L}}dt+
\end{equation*}%
\begin{equation*}
\underset{0}{\overset{T}{\int }}\left\langle \underset{j=1}{\overset{3}{\sum 
}}u_{Lj}D_{j}u_{L},\phi (t)v\right\rangle _{\Omega _{L}}dt=\underset{0}{%
\overset{T}{\int }}\left\langle f_{L},\phi (t)v\right\rangle _{\Omega
_{L}}dt+\left\langle u_{L}\left( 0\right) ,\phi (0)v\right\rangle _{\Omega
_{L}}.
\end{equation*}

If we compare this with (\ref{5.8}) after replacing $w_{j}$ with any $v\in
V\left( \Omega _{L}\right) $ then we obtain 
\begin{equation*}
\phi (0)\left\langle u_{L}\left( 0\right) -u_{0L},v\right\rangle _{\Omega
_{L}}=0.
\end{equation*}%
Whence, we get the realisation of the initial condition by virtue of
arbitrariness of $v\in V\left( \Omega _{L}\right) $ and $\phi $, since
function $\phi $ one can choose as $\phi (0)\neq 0$.

Consequently, the following result is proven.

\begin{theorem}
\label{Th_2.1}Under above mentioned conditions for any 
\begin{equation*}
u_{0L}\in \left( H\left( \Omega _{L}\right) \right) ^{3},\quad f_{L}\in
L^{2}\left( 0,T;V^{\ast }\left( \Omega _{L}\right) \right)
\end{equation*}%
problem (\ref{3.3}) - (\ref{3.5}) has weak solution $u_{L}\left( t,x\right) $
that belongs to $\mathcal{V}\left( Q_{L}^{T}\right) $.
\end{theorem}

\begin{remark}
From the obtained a priori estimates and Propositions \ref{P_3.1} and \ref%
{P_3.2} follows of the fulfilment of all conditions of the general theorem
of the compactness method (see, e. g. \cite{Sol3}, \cite{SolAhm}, and for
complementary informations see, \cite{Sol2}, \cite{Sol4}). Consequently, one
could be to study the solvability of problem (\ref{3.3}) - (\ref{3.5}) with
use of this general theorem.
\end{remark}

\section{\label{Sec_I.5}Uniqueness of Solution of Problem (3.3) - (3.5)}

For the study of the uniqueness of the solution as usually: we will assume
that posed problem have, at least, two different solutions $u=\left(
u_{1},u_{2},u_{3}\right) $, $v=\left( v_{1},v_{2},v_{3}\right) $. Below will
show that this isn't possible, and for which one need to investigate their
difference, i.e. $w=u-v$.\ (Here for brevity we won't specify indexes for
functions, which showing that here is investigated the system of equations (%
\ref{3.3}) - (\ref{3.5}) on $Q_{L}^{T}$.)

So, we obtain the following problem for $w=u-v$ 
\begin{equation*}
\frac{\partial w}{\partial t}-\nu \left[ \left( 1+a_{1}^{-2}\right)
D_{1}^{2}+\left( 1+a_{2}^{-2}\right) D_{2}^{2}\right] w-2\nu
a_{1}^{-1}a_{2}^{-1}D_{1}D_{2}w+
\end{equation*}%
\begin{equation*}
\left( u_{1}-a_{1}^{-1}u_{3}\right) D_{1}u-\left(
v_{1}-a_{1}^{-1}v_{3}\right) D_{1}v+\left( u_{2}-a_{2}^{-1}u_{3}\right)
D_{2}u-
\end{equation*}%
\begin{equation*}
\left( v_{2}-a_{2}^{-1}v_{3}\right) D_{2}v=0,
\end{equation*}%
\begin{equation*}
\func{div}w=D_{1}\left[ \left( u-a_{1}^{-1}u_{3}\right) -\left(
v-a_{1}^{-1}v_{3}\right) \right] +D_{2}\left[ \left(
u-a_{2}^{-1}u_{3}\right) \right. -
\end{equation*}%
\begin{equation}
\left. \left( v-a_{2}^{-1}v_{3}\right) \right] =D_{1}w+D_{2}w-\left( a_{1}^{-1}D_{1}+a_{2}^{-1}D_{2}\right) w_{3}=0,
\label{3.18}
\end{equation}%
\begin{equation}
w\left( 0,x\right) =0,\quad x\in \Omega \cap L;\quad w\left\vert \ _{\left( 0,T\right) \times \partial \Omega _{L}}\right. =0.
\label{3.19}
\end{equation}

Hence we derive 
\begin{equation*}
\frac{1}{2}\frac{d}{dt}\left\Vert w\right\Vert _{2}^{2}+\nu \left[ \left(
1+a_{1}^{-2}\right) \left\Vert D_{1}w\right\Vert _{2}^{2}+\left(
1+a_{2}^{-2}\right) \left\Vert D_{2}w\right\Vert _{2}^{2}\right] +
\end{equation*}%
\begin{equation*}
2\nu a_{1}^{-1}a_{2}^{-1}\left\langle D_{1}w,D_{2}w\right\rangle _{\Omega
_{L}}+\left\langle \left( u_{1}-a_{1}^{-1}u_{3}\right) D_{1}u-\left(
v_{1}-a_{1}^{-1}v_{3}\right) D_{1}v,w\right\rangle _{\Omega _{L}}+
\end{equation*}%
\begin{equation*}
\left\langle \left( u_{2}-a_{2}^{-1}u_{3}\right) D_{2}u-\left(
v_{2}-a_{2}^{-1}v_{3}\right) D_{2}v,w\right\rangle _{\Omega _{L}}=0
\end{equation*}%
or 
\begin{equation*}
\frac{1}{2}\frac{d}{dt}\left\Vert w\right\Vert _{2}^{2}+\nu \left(
\left\Vert D_{1}w\right\Vert _{2}^{2}+\left\Vert D_{2}w\right\Vert
_{2}^{2}\right) +\nu \left[ a_{1}^{-2}\left\Vert D_{1}w\right\Vert
_{2}^{2}+a_{2}^{-2}\left\Vert D_{2}w\right\Vert _{2}^{2}+\right.
\end{equation*}%
\begin{equation*}
\left. 2a_{1}^{-1}a_{2}^{-1}\left\langle D_{1}w,D_{2}w\right\rangle _{\Omega
_{L}}\right] +\left\langle u_{1}D_{1}u-v_{1}D_{1}v,w\right\rangle _{\Omega
_{L}}+\left\langle u_{2}D_{2}u-v_{2}D_{2}v,w\right\rangle _{\Omega _{L}}-
\end{equation*}%
\begin{equation}
a_{1}^{-1}\left\langle u_{3}D_{1}u-v_{3}D_{1}v,w\right\rangle _{\Omega
_{L}}-a_{2}^{-1}\left\langle u_{3}D_{2}u-v_{3}D_{2}v,w\right\rangle _{\Omega
_{L}}=0.  \label{3.20}
\end{equation}

If the last 4 terms in the sum of left part (\ref{3.20}) consider separately
and if these simplify by calculations then we get 
\begin{equation*}
\left\langle w_{1}D_{1}u,w\right\rangle _{\Omega _{L}}+\left\langle
v_{1}D_{1}w,w\right\rangle _{\Omega _{L}}+\left\langle
w_{2}D_{2}u,w\right\rangle _{\Omega _{L}}+\left\langle
v_{2}D_{2}w,w\right\rangle _{\Omega _{L}}-
\end{equation*}%
\begin{equation*}
a_{1}^{-1}\left\langle w_{3}D_{1}u,w\right\rangle _{\Omega
_{L}}-a_{1}^{-1}\left\langle v_{3}D_{1}w,w\right\rangle _{\Omega
_{L}}-a_{2}^{-1}\left\langle w_{3}D_{2}u,w\right\rangle _{\Omega
_{L}}-a_{2}^{-1}\left\langle v_{3}D_{2}w,w\right\rangle _{\Omega _{L}}=
\end{equation*}%
\begin{equation*}
\left\langle w_{1}D_{1}u,w\right\rangle _{\Omega _{L}}+\frac{1}{2}%
\left\langle v_{1},D_{1}w^{2}\right\rangle _{\Omega _{L}}+\left\langle
w_{2}D_{2}u,w\right\rangle _{\Omega _{L}}+\frac{1}{2}\left\langle
v_{2},D_{2}w^{2}\right\rangle _{\Omega _{L}}-
\end{equation*}%
\begin{equation*}
a_{1}^{-1}\left\langle w_{3}D_{1}u,w\right\rangle _{\Omega _{L}}-\frac{1}{2}%
a_{1}^{-1}\left\langle v_{3},D_{1}w^{2}\right\rangle _{\Omega
_{L}}-a_{2}^{-1}\left\langle w_{3}D_{2}u,w\right\rangle _{\Omega _{L}}-
\end{equation*}%
\begin{equation*}
\frac{1}{2}a_{2}^{-1}\left\langle v_{3},D_{2}w^{2}\right\rangle =\frac{1}{2}%
\left\langle v_{1}-a_{1}^{-1}v_{3},D_{1}w^{2}\right\rangle _{\Omega _{L}}+%
\frac{1}{2}\left\langle v_{2}-a_{2}^{-1}v_{3},D_{2}w^{2}\right\rangle
_{\Omega _{L}}+
\end{equation*}%
\begin{equation*}
\left\langle \left( w_{1}-a_{1}^{-1}w_{3}\right) w,D_{1}u\right\rangle
_{\Omega _{L}}+\left\langle \left( w_{2}-a_{2}^{-1}w_{3}\right)
w,D_{2}u\right\rangle _{\Omega _{L}}=
\end{equation*}%
\begin{equation*}
\left\langle \left( w_{1}-a_{1}^{-1}w_{3}\right) w,D_{1}u\right\rangle
_{\Omega _{L}}+\left\langle \left( w_{2}-a_{2}^{-1}w_{3}\right)
w,D_{2}u\right\rangle _{\Omega _{L}}.
\end{equation*}%
In the last equality were used the equation $\func{div}v=0$ (see, (\ref{3.4}%
)) and the condition (\ref{3.19}).

If takes into account this equality in equation (\ref{3.20}) then we get 
\begin{equation*}
\frac{1}{2}\frac{d}{dt}\left\Vert w\right\Vert _{2}^{2}+\nu \left(
\left\Vert D_{1}w\right\Vert _{2}^{2}+\left\Vert D_{2}w\right\Vert
_{2}^{2}\right) +\nu \left[ a_{1}^{-2}\left\Vert D_{1}w\right\Vert
_{2}^{2}+\right.
\end{equation*}%
\begin{equation*}
\left. a_{2}^{-2}\left\Vert D_{2}w\right\Vert
_{2}^{2}+2a_{1}^{-1}a_{2}^{-1}\left\langle D_{1}w,D_{2}w\right\rangle
_{\Omega _{L}}\right] +\left\langle \left( w_{1}-a_{1}^{-1}w_{3}\right)
w,D_{1}u\right\rangle _{\Omega _{L}}+
\end{equation*}%
\begin{equation}
\left\langle \left( w_{2}-a_{2}^{-1}w_{3}\right) w,D_{2}u\right\rangle
_{\Omega _{L}}=0,\quad \left( t,x\right) \in \left( 0,T\right) \times \Omega
_{L}.  \label{3.21}
\end{equation}

Thus we obtain the Cauchy problem for equation (\ref{3.21}) with the initial
condition 
\begin{equation}
\left\Vert w\right\Vert _{2}\left( 0\right) =0.  \label{3.22}
\end{equation}

We get the following Cauchy problem for the differential inequation using
the appropriate estimates 
\begin{equation*}
\frac{1}{2}\frac{d}{dt}\left\Vert w\right\Vert _{2}^{2}+\nu \left(
\left\Vert D_{1}w\right\Vert _{2}^{2}+\left\Vert D_{2}w\right\Vert
_{2}^{2}\right) \leq
\end{equation*}%
\begin{equation}
\left\vert \left\langle \left( w_{1}-a_{1}^{-1}w_{3}\right)
w,D_{1}u\right\rangle _{\Omega _{L}}\right\vert +\left\vert \left\langle
\left( w_{2}-a_{2}^{-1}w_{3}\right) w,D_{2}u\right\rangle _{\Omega
_{L}}\right\vert ,  \label{3.21'}
\end{equation}

with the initial condition (\ref{3.22}).

Then for the right side of (\ref{3.21'}) we get the following estimate 
\begin{equation*}
\left\vert \left\langle \left( w_{1}-a_{1}^{-1}w_{3}\right)
w,D_{1}u\right\rangle _{\Omega _{L}}\right\vert +\left\vert \left\langle
\left( w_{2}-a_{2}^{-1}w_{3}\right) w,D_{2}u\right\rangle _{\Omega
_{L}}\right\vert \leq
\end{equation*}%
\begin{equation*}
\left( \left\Vert w_{1}-a_{1}^{-1}w_{3}\right\Vert _{4}+\left\Vert
w_{2}-a_{2}^{-1}w_{3}\right\Vert _{4}\right) \left\Vert w\right\Vert
_{4}\left\Vert \nabla u\right\Vert _{2}\leq
\end{equation*}%
whence we derive 
\begin{equation*}
\left( 1+\max \left\{ \left\vert a_{1}^{-1}\right\vert ,\left\vert
a_{2}^{-1}\right\vert \right\} \right) \left\Vert w\right\Vert
_{4}^{2}\left\Vert \nabla u\right\Vert _{2}\leq c\left\Vert w\right\Vert
_{2}\left\Vert \nabla w\right\Vert _{2}\left\Vert \nabla u\right\Vert _{2}
\end{equation*}%
thanks of Gagliardo-Nirenberg inequality (\cite{BesIlNik}).

It need to note that 
\begin{equation*}
\left( w_{1}-a_{1}^{-1}w_{3}\right) w,\ \left( w_{2}-a_{2}^{-1}w_{3}\right)
w\in L^{2}\left( 0,T;V^{\ast }\left( \Omega _{L}\right) \right) ,
\end{equation*}%
by virtue of (\ref{3.16}).

Now taking into account this in (\ref{3.21'}) one can arrive the following
Cauchy problem for differential inequation 
\begin{equation*}
\frac{1}{2}\frac{d}{dt}\left\Vert w\right\Vert _{2}^{2}\left( t\right) +\nu
\left\Vert \nabla w\right\Vert _{2}^{2}\left( t\right) \leq c\left\Vert
w\right\Vert _{2}\left( t\right) \left\Vert \nabla w\right\Vert _{2}\left(
t\right) \left\Vert \nabla u\right\Vert _{2}\left( t\right) \leq
\end{equation*}%
\begin{equation*}
C\left( c,\nu \right) \left\Vert \nabla u\right\Vert _{2}^{2}\left( t\right)
\left\Vert w\right\Vert _{2}^{2}\left( t\right) +\nu \left\Vert \nabla
w\right\Vert _{2}^{2}\left( t\right) ,\quad \left\Vert w\right\Vert
_{2}\left( 0\right) =0,
\end{equation*}%
since $w\in L^{\infty }\left( 0,T;H\left( \Omega _{L}\right) \right) $.
Consequently, $\left\Vert w\right\Vert _{2}\left\Vert \nabla w\right\Vert
_{2}\in L^{2}\left( 0,T\right) $ by virtue of the proved above existence
theorem $w\in \mathcal{V}\left( Q_{L}^{T}\right) $, where $C\left( c,\nu
\right) >0$ is constant.

Thus we obtain the problem 
\begin{equation*}
\frac{d}{dt}\left\Vert w\right\Vert _{2}^{2}\left( t\right) \leq 2C\left(
c,\nu \right) \left\Vert \nabla u\right\Vert _{2}^{2}\left( t\right)
\left\Vert w\right\Vert _{2}^{2}\left( t\right) ,\quad \left\Vert
w\right\Vert _{2}\left( 0\right) =0.
\end{equation*}%
If to denote $\left\Vert w\right\Vert _{2}^{2}\left( t\right) \equiv y\left(
t\right) $ then 
\begin{equation*}
\frac{d}{dt}y\left( t\right) \leq 2C\left( c,\nu \right) \left\Vert \nabla
u\right\Vert _{2}^{2}\left( t\right) y\left( t\right) ,\quad y\left(
0\right) =0.
\end{equation*}

Whence follows $\left\Vert w\right\Vert _{2}^{2}\left( t\right) \equiv
y\left( t\right) =0$, and consequently the following result is proven: \ 

\begin{theorem}
\label{Th_2.2}Under above mentioned conditions for any 
\begin{equation*}
\left( f,u_{0}\right) \in L^{2}\left( 0,T;V^{\ast }\left( \Omega _{L}\right)
\right) \times H\left( \Omega _{L}\right)
\end{equation*}%
problem (\ref{3.3}) - (\ref{3.5}) has a unique weak solution $u\left(
t,x\right) $ that is contained in $\mathcal{V}\left( Q_{L}^{T}\right) $.
\end{theorem}

\section{\label{Sec_I.6}Proof of Theorem \protect\ref{Th_1}}

\begin{proof}
(of Theorem \ref{Th_1}). As were noted in introduction, under the above
mentioned conditions problem (1.1%
${{}^1}$%
) - (\ref{3}) is weakly solvable and any solution belongs to the space $%
\mathcal{V}\left( Q^{T}\right) $. Consequently, under the conditions of
Theorem \ref{Th_1} this problem also has weak solution that belongs, at
least, to the space $\mathcal{V}\left( Q^{T}\right) $. But as shown in
Sections \ref{Sec_I.4} under conditions of Theorem \ref{Th_1} the auxiliary
problems of problem (1.1%
${{}^1}$%
) - (\ref{3}) are weakly solvable and any solution belongs to the space $%
\mathcal{V}\left( Q_{L}^{T}\right) $. Moreover, as shown in Section \ref%
{Sec_I.5} weak solution of each of these problems is unique. Hence follows,
that we can employ of Lemma \ref{L_2.2} to solutions of problem (1.1%
${{}^1}$%
) - (\ref{3}) on $Q_{L}^{T}$ due of the smoothness of solutions of this
problem.

So, assume problem (1.1%
${{}^1}$%
) - (\ref{3}) has, at least, two different weak solutions under conditions
of Theorem \ref{Th_1}. It is clear that if the problem have more than one
solution then there is, at least, some subdomain of $Q^{T}\equiv \left(
0,T\right) \times \Omega $, on which this problem have, at least, two
solutions that different.Consequently, starting from the above Lemma \ref%
{L_2.2} is sufficiently to investigate the existence and uniqueness of the
posed problem on arbitrary fixed subdomain in order to shows that exist or
unexist such subdomens, on which the studied problem can possess more than
one solutions. More exactly it is sufficiently to study of this question in
the case when subdomains are generated by arbitrary fixed hyperplanes by
virtue of Lemma \ref{L_2.2}. For this aim it is enough to prove, that isn't
exist such subdomains, on which the problem (1.1%
${{}^1}$%
) - (\ref{3}) could has of more than one solution by virtue of Lemma \ref%
{L_2.2}. Thus, in order to end of the proof is remains to use the above
results (i.e. Theorems \ref{Th_2.1} and \ref{Th_2.2}).

Indeed, as follows from theorems that were proved in the previous sections
there not are exist subdomains, on which the problem (1.1%
${{}^1}$%
) - (\ref{3}) could be possesses more than one weak solution.

Consequently, according of Lemma \ref{L_2.2} we obtain, that the problem (1.1%
${{}^1}$%
) - (\ref{3}) under conditions of Theorem \ref{Th_1} possesses only one weak
solution.
\end{proof}

Whence can make the following conclusion.

\subsection{Conclusion}

Let's 
\begin{equation*}
f\in L^{2}\left( 0,T;V^{\ast }\left( \Omega \right) \right) ,\ u_{0}\in
H\left( \Omega \right) .
\end{equation*}%
It well-known that following inclusions are dense 
\begin{equation*}
L^{2}\left( 0,T;H^{1/2}\left( \Omega \right) \right) \subset L^{2}\left(
Q^{T}\right) ;\ H^{1/2}\left( \Omega \right) \subset H\left( \Omega \right)
\ \ \&
\end{equation*}%
\begin{equation*}
L^{2}\left( 0,T;H^{1/2}\left( \Omega \right) \right) \subset L^{2}\left(
0,T;H^{-1}\left( \Omega \right) \right) .
\end{equation*}%
Hence, there exist such sequences 
\begin{equation*}
\left\{ u_{0m}\right\} _{m=1}^{\infty }\subset H^{1/2}\left( \Omega \right)
;\left\{ f_{m}\right\} _{m=1}^{\infty }\subset L^{2}\left( 0,T;H^{1/2}\left(
\Omega \right) \right)
\end{equation*}%
that $u_{0m}\longrightarrow u_{0}$ in $H\left( \Omega \right) $ , $%
f_{m}\longrightarrow f$ in $L^{2}\left( 0,T;H^{-1}\left( \Omega \right)
\right) $.

Thus, we establish following result.

\begin{theorem}
\label{Th_8}Let $\Omega $ be a Lipschitz open bounded domain in $R^{3}$ and
the given functions $f$ and $u_{0}$\ satisfy of conditions $f$ $\in
L^{2}\left( 0,T;H^{1/2}\left( \Omega \right) \right) $ and $u_{0}\in
H^{1/2}\left( \Omega \right) $, respectively. Then there exists unique
function $u\in \mathcal{V}\left( Q^{T}\right) $ that is the weak solution of
the considered problem, in the sense of Definition \ref{D_2.2}.
\end{theorem}

Roughly speaking, since $L^{2}\left( 0,T;H^{1/2}\left( \Omega \right)
\right) $ and $H^{1/2}\left( \Omega \right) $ are everywhere dense in spaces 
$L^{2}\left( 0,T;H^{1/2}\left( \Omega \right) \right) $ and $H^{1/2}\left(
\Omega \right) $, respectively, then if functions $f$ and $u_{0}$ are any
given functions from $L^{2}\left( 0,T;V^{\ast }\left( \Omega \right) \right) 
$ and $H\left( \Omega \right) $, respectively then in their any neighbohoods
there are functions $\widetilde{f}$ and $\widetilde{u}_{0}$ from $%
L^{2}\left( 0,T;H^{1/2}\left( \Omega \right) \right) $ and $H^{1/2}\left(
\Omega \right) $, respectively that the problem (1.1%
${{}^1}$%
) - (\ref{3}) has unique weak solution $u$, that belongs to a bounded subset
of $\mathcal{V}\left( Q^{T}\right) $, where a weak solution be understood in
the sense of Definition \ref{D_2.2}.

So, under conditions of Theorem \ref{Th_1} the uniqueness of weak solution $%
u(x,t)$ (of velocity vector) of the problem (1.1%
${{}^1}$%
) - (\ref{3}) obtained from the mixed problem for the incompressible
Navier-Stokes $3D$-equation proved (explanations of the last proposition see
Notation \ref{N_1} and next paragraph of this Notation \ref{N_1}).

\part{\label{Part II}Employment of modified approach to study of uniqueness}

\section{\label{Sec_II.7}One conditional uniqueness theorem for problem (1.1$%
^{1}$) - (1.3)}

We believed there have the sense to provide here yet one result connected
with same question for problem (\ref{1}) - (\ref{3}), but with conditions
onto the given functions under which the existence theorem of the weak
solution of this problem is proven. Here the known approach for the
investigation of the uniqueness of solution of problem (1.1%
${{}^1}$%
) - (\ref{3}) is applied, but with use also other properties of this problem.

Let posed problem have two different solutions: $u,v\in \mathcal{V}\left(
Q^{T}\right) $, then within known approach we get the following problem for
vector function $w(t,x)=u(t,x)-v(t,x)$ 
\begin{equation}
\frac{1}{2}\frac{\partial }{\partial t}\left\Vert w\right\Vert _{2}^{2}+\nu \left\Vert \nabla w\right\Vert _{2}^{2}+\underset{j,k=1}{\overset{3}{\sum }}\left\langle \frac{\partial v_{k}}{\partial x_{j}}w_{k},w_{j}\right\rangle =0,
\label{2.9}
\end{equation}%
\begin{equation}
w\left( 0,x\right) =w_{0}\left( x\right) =0,\quad x\in \Omega ;\quad
w\left\vert \ _{\left[ 0,T\right] \times \partial \Omega }=0\right. ,
\label{2.10}
\end{equation}%
where $\Omega \subset 
\mathbb{R}
^{3}$ is above-mentioned domain.\ 

So, for the proof of triviality of solution of problem (\ref{2.9})-(\ref%
{2.10}), as usually will used method of contradiction. Consequently, one
will start with assume that problem have nontrivial solution.

In addition, it is need to noted here will used of the peculiarity of having
nonlinearity of this problem.

In the beginning we will study the following quadratic form (\cite{Gan}) for
examination of problem (\ref{2.9})-(\ref{2.10}) 
\begin{equation*}
B\left( w,w\right) =\underset{j,k=1}{\overset{3}{\sum }}\left( \frac{%
\partial v_{k}}{\partial x_{j}}w_{k}w_{j}\right) \left( t,x\right) ,
\end{equation*}%
denote it as 
\begin{equation*}
B\left( w,w\right) \equiv \underset{j,k=1}{\overset{3}{\sum }}\left(
a_{jk}w_{k}w_{j}\right) \left( t,x\right) .
\end{equation*}%
It is clear that behavior of the surface generated by function $B\left(
w,w\right) $ respect to the variables $w_{k},\ k=1,2,3$ depende of the
accelerations of the flow on the different directions.

Consider the question: it would possible to transform the quadratic form $%
B\left( w,w\right) $ to the canonical form, namely to the following form 
\begin{eqnarray*}
B\left( w,w\right) &\equiv &\underset{i=1}{\overset{3}{\sum }}\left(
b_{i}w_{i}^{2}\right) \left( t,x\right) ,\quad b_{i}\left( t,x\right) \equiv
b_{i}\left( \overline{D_{j}v_{k}}\right) , \\
\text{where \ }D_{i}v_{k} &\equiv &\frac{\partial v_{k}}{\partial x_{i}}%
,\quad i,k=1,2,3,\quad b_{i}:%
\mathbb{R}
^{9}\longrightarrow 
\mathbb{R}
\text{ be functions?}
\end{eqnarray*}

The matrix $\left\Vert a_{jk}\right\Vert $ of coefficients of the quadratic
form $B\left( w,w\right) $\ can be represented in the following form 
\begin{equation*}
\left\Vert a_{jk}\right\Vert _{j,k=1}^{3}=\left\Vert 
\begin{array}{ccc}
a_{11} & a_{12} & a_{13} \\ 
a_{21} & a_{22} & a_{23} \\ 
a_{31} & a_{32} & a_{33}%
\end{array}%
\right\Vert ,\quad \text{where }a_{jk}=a_{kj}=\frac{1}{2}\left(
D_{j}v_{k}+D_{k}v_{j}\right) ,
\end{equation*}%
hence it is symmetric matrix. As known in this case the above transformation
exists according of symmetricness of matrix $\left\Vert a_{jk}\right\Vert
_{j,k=1}^{3}$ (see, \cite{Gan}). Consequently, coefficients $b_{i}$ are
defined and have the following presentations 
\begin{equation*}
b_{1}=D_{1}v_{1};\ b_{2}=D_{2}v_{2}-\frac{\left(
D_{1}v_{2}+D_{2}v_{1}\right) ^{2}}{4b_{1}};\ b_{3}=\frac{\det \left\Vert
D_{i}v_{k}\right\Vert _{i,k=1}^{3}}{\det \left\Vert D_{i}v_{k}\right\Vert
_{i,k=1}^{2}},
\end{equation*}%
where $\left\Vert D_{i}v_{k}\right\Vert _{i,k=1}^{3}$\ and $\left\Vert
D_{i}v_{k}\right\Vert _{i,k=1}^{2}$ define by equalities

\begin{center}
$\left\Vert D_{i}v_{k}\right\Vert _{i,k=1}^{3}\equiv \left\Vert 
\begin{array}{ccc}
D_{1}v_{1} & \frac{1}{2}\left( D_{1}v_{2}+D_{2}v_{1}\right) & \frac{1}{2}%
\left( D_{1}v_{3}+D_{3}v_{1}\right) \\ 
\frac{1}{2}\left( D_{1}v_{2}+D_{2}v_{1}\right) & D_{2}v_{2} & \frac{1}{2}%
\left( D_{2}v_{3}+D_{3}v_{2}\right) \\ 
\frac{1}{2}\left( D_{1}v_{3}+D_{3}v_{1}\right) & \frac{1}{2}\left(
D_{2}v_{3}+D_{3}v_{2}\right) & D_{3}v_{3}%
\end{array}%
\right\Vert $
\end{center}

and

\begin{center}
$\left\Vert D_{i}v_{k}\right\Vert _{i,k=1}^{2}\equiv \left\Vert 
\begin{array}{cc}
D_{1}v_{1} & \frac{1}{2}\left( D_{1}v_{2}+D_{2}v_{1}\right) \\ 
\frac{1}{2}\left( D_{1}v_{2}+D_{2}v_{1}\right) & D_{2}v_{2}%
\end{array}%
\right\Vert $ \ 
\end{center}

for any $\left( t,x\right) \in Q^{T}\equiv \left( 0,T\right) \times \Omega $.

Therefore we have 
\begin{equation}
\left( B\left( w,w\right) \right) \left( t,x\right) \equiv \underset{j,k=1}{%
\overset{3}{\sum }}\left( a_{jk}w_{k}w_{j}\right) \left( t,x\right) \equiv 
\underset{j=1}{\overset{3}{\sum }}b_{j}\left( t,x\right) \cdot
w_{j}^{2}\left( t,x\right)  \label{2.11}
\end{equation}

that one can rewrite in the following open form 
\begin{equation*}
B\left( w,w\right) \equiv \frac{1}{D_{1}v_{1}}\left[ 2D_{1}v_{1}w_{1}+\left(
D_{1}v_{2}+D_{2}v_{1}\right) w_{2}+\left( D_{1}v_{3}+D_{3}v_{1}\right) w_{3}%
\right] ^{2}+
\end{equation*}%
\begin{equation*}
\frac{1}{\left( 4D_{1}v_{1}\right) ^{2}}\left( 4D_{1}v_{1}D_{2}v_{2}-\left(
D_{1}v_{2}+D_{2}v_{1}\right) ^{2}\right) \times
\end{equation*}%
\begin{equation*}
\left[ \left( 4D_{1}v_{1}D_{2}v_{2}-\left( D_{1}v_{2}+D_{2}v_{1}\right)
^{2}\right) w_{2}\right. +
\end{equation*}%
\begin{equation*}
\left. \left( 2D_{1}v_{1}\left( D_{2}v_{3}+D_{3}v_{2}\right) -\left(
D_{1}v_{2}+D_{2}v_{1}\right) \left( D_{1}v_{3}+D_{3}v_{1}\right) \right)
w_{3}\right] ^{2}+
\end{equation*}%
\begin{equation*}
\frac{1}{4}\left[ 4D_{1}v_{1}D_{2}v_{2}D_{3}v_{3}+\left(
D_{1}v_{2}+D_{2}v_{1}\right) \left( D_{1}v_{3}+D_{3}v_{1}\right) \left(
D_{2}v_{3}+D_{3}v_{2}\right) \right. -
\end{equation*}%
\begin{equation*}
\left. D_{1}v_{1}\left( D_{2}v_{3}+D_{3}v_{2}\right) ^{2}-D_{2}v_{2}\left(
D_{1}v_{3}+D_{3}v_{1}\right) ^{2}-D_{3}v_{3}\left(
D_{1}v_{2}+D_{2}v_{1}\right) ^{2}\right] w_{3}^{2}.
\end{equation*}

If take account (\ref{2.11}) in the equation (\ref{2.9}) then we get 
\begin{equation*}
\frac{1}{2}\frac{\partial }{\partial t}\left\Vert w\right\Vert _{2}^{2}+\nu
\left\Vert \nabla w\right\Vert _{2}^{2}+\underset{j=1}{\overset{3}{\sum }}%
\left\langle b_{j}w_{j},w_{j}\right\rangle =0,\quad \left\Vert
w_{0}\right\Vert _{2}=0,
\end{equation*}%
or 
\begin{equation}
\frac{1}{2}\frac{\partial }{\partial t}\left\Vert w\right\Vert _{2}^{2}=-\nu
\left\Vert \nabla w\right\Vert _{2}^{2}-\underset{j=1}{\overset{3}{\sum }}%
\left\langle b_{j}w_{j},w_{j}\right\rangle ,\quad \left\Vert
w_{0}\right\Vert _{2}=0.  \label{2.12}
\end{equation}

This shows that if $b_{j}\left( t,x\right) \geq 0$ for a.e. $\left(
t,x\right) \in Q^{T}$ then the posed problem have unique solution. It is
need noted that images of functions $b_{j}\left( t,x\right) $ and $%
D_{i}v_{k} $ belong to the bounded subset of the same space.

So, is remains to investigate the cases when the mentioned isn't fulfill.

Here the following variants are possible:

1. Integral of $B\left( w,w\right) $ is determined and non-negative 
\begin{equation*}
\underset{\Omega }{\int }B\left( w,w\right) dx=\underset{j=1}{\overset{3}{%
\sum }}\left\langle b_{j}w_{j},w_{j}\right\rangle \equiv \underset{j=1}{%
\overset{3}{\sum }}{}\underset{\Omega }{\int }b_{j}w_{j}^{2}dx\geq 0;
\end{equation*}

In this case one can conclude the main problem have unique solution (and
this solution is stable).

2. Integral of is undetermined and $\underset{j=1}{\overset{3}{\sum }}{}%
\underset{\Omega }{\int }b\ w_{j}^{2}dx\neq 0$.

In this case for investigation of problem (\ref{2.12}) it is necessary to
derive suitable estimates for $B\left( w,w\right) \equiv \underset{j,k=1}{%
\overset{3}{\sum }}\left( D_{i}v_{k}w_{k}w_{j}\right) $.

So, let $\underset{\Omega }{\int }B\left( w,w\right) dx$ is undetermined.
Therefore we need estimate the right part of the equation from (\ref{2.12})

\begin{equation*}
\frac{1}{2}\frac{\partial }{\partial t}\left\Vert w\right\Vert _{2}^{2}=-\nu
\left\Vert \nabla w\right\Vert _{2}^{2}+\left\vert \underset{j,k=1}{\overset{%
3}{\sum }}\left\langle D_{i}v_{k}w_{k},w_{j}\right\rangle \right\vert \leq
\end{equation*}%
\begin{equation}
-\underset{j=1}{\overset{3}{\sum }}{}\underset{\Omega }{\int }\nu \left\vert
\nabla w_{j}\left( t,x\right) \right\vert ^{2}dx+\underset{j,k=1}{\overset{3}%
{\sum }}{}\underset{\Omega }{\int }\left\vert \left(
D_{i}v_{k}w_{k}w_{j}\right) \left( t,x\right) \right\vert dx,  \label{2.13}
\end{equation}%
more precisely, we need estimate the second adding in the right part of (\ref%
{2.13}). So, for one of the trilinear terms we obtain \footnote{%
It is known that (\cite{Lad1}, \cite{Lio1}) $\left\vert \left\langle
u_{k}D_{i}v_{j},w_{l}\right\rangle \right\vert \leq \left\Vert
u_{k}\right\Vert _{q}\left\Vert D_{i}v_{j}\right\Vert _{2}\left\Vert
w_{l}\right\Vert _{n},\quad n\geq 3;$%
\par
$\left\Vert v_{j}\right\Vert _{4}\leq C\left( mes\ \Omega \right) \left\Vert
Dv_{j}\right\Vert _{2}^{\frac{1}{2}}\left\Vert v_{j}\right\Vert _{2}^{\frac{1%
}{2}},\quad n=2$} 
\begin{equation*}
\left\vert \left\langle D_{i}v_{j}w_{i},w_{j}\right\rangle \right\vert \leq
\left\Vert D_{i}v_{j}\right\Vert _{2}\left\Vert w_{i}\right\Vert
_{p_{1}}\left\Vert w_{j}\right\Vert _{p_{2}},
\end{equation*}%
with use of the H\={o}lder inequality, where is sufficient to choose, $%
p_{1}=p_{2}=4$. Consequently, one can estimate $\underset{\Omega }{\int }%
B\left( w,w\right) dx$ as follows 
\begin{equation*}
\underset{\Omega }{\int }\left\vert B\left( w,w\right) \right\vert dx\leq 
\underset{i,j=1}{\overset{3}{\sum }}\left\Vert D_{j}v_{i}\right\Vert
_{2}\left\Vert w_{i}\right\Vert _{4}\left\Vert w_{j}\right\Vert _{4}.
\end{equation*}

{}Hence, use Gagliardo-Nirenberg inequality (see, e.g., \cite{BesIlNik}) we
get 
\begin{equation*}
\left\Vert w_{j}\right\Vert _{4}\leq c\left\Vert w_{j}\right\Vert
_{2}^{1-\sigma }\left\Vert \nabla w_{j}\right\Vert _{2}^{\sigma },\quad
\sigma =\frac{3}{4},
\end{equation*}%
where $c\equiv C\left( 4,2,2,0,1\right) $, and for this case 
\begin{equation*}
\left\Vert w_{j}\right\Vert _{4}\leq c\left\Vert w_{j}\right\Vert _{2}^{%
\frac{1}{4}}\left\Vert \nabla w_{j}\right\Vert _{2}^{\frac{3}{4}%
}\Longrightarrow \left\Vert w_{j}\right\Vert _{4}^{2}\leq c^{2}\left\Vert
w_{j}\right\Vert _{2}^{\frac{1}{2}}\left\Vert \nabla w_{j}\right\Vert _{2}^{%
\frac{3}{2}}.
\end{equation*}%
Therefore 
\begin{equation*}
\underset{\Omega }{\int }\left\vert B\left( w,w\right) \right\vert dx\leq
c^{2}\underset{i,j=1}{\overset{3}{\sum }}\left\Vert D_{j}v_{i}\right\Vert
_{2}\left\Vert w_{i}\right\Vert _{2}^{\frac{1}{4}}\left\Vert \nabla
w_{i}\right\Vert _{2}^{\frac{3}{4}}\left\Vert w_{j}\right\Vert _{2}^{\frac{1%
}{4}}\left\Vert \nabla w_{j}\right\Vert _{2}^{\frac{3}{4}}
\end{equation*}%
holds. Now taking into account the above estimate in (\ref{2.14}) we derive 
\begin{equation*}
\frac{1}{2}\frac{\partial }{\partial t}\left\Vert w\left( t\right)
\right\Vert _{2}^{2}\leq -\underset{j=1}{\overset{3}{\sum }}{}\nu \left\Vert
\nabla w_{j}\left( t\right) \right\Vert _{2}^{2}+c^{2}\underset{i,j=1}{%
\overset{3}{\sum }}{}\left\Vert D_{j}v_{i}\left( t\right) \right\Vert
_{2}\left\Vert w_{i}\left( t\right) \right\Vert _{2}^{\frac{1}{2}}\left\Vert
\nabla w_{i}\left( t\right) \right\Vert _{2}^{\frac{3}{2}}
\end{equation*}%
\begin{equation*}
\leq -\underset{j=1}{\overset{3}{\sum }}\left\Vert \nabla w_{j}\left(
t\right) \right\Vert _{2}^{\frac{3}{2}}\left[ \nu \left\Vert \nabla
w_{j}\left( t\right) \right\Vert _{2}^{\frac{1}{2}}-c^{2}\underset{i=1}{%
\overset{3}{\sum }}\left\Vert D_{i}v_{j}\left( t\right) \right\Vert
_{2}\left\Vert w_{j}\left( t\right) \right\Vert _{2}^{\frac{1}{2}}\right]
\end{equation*}%
\begin{equation*}
\leq -\underset{j=1}{\overset{n}{\sum }}\left\Vert \nabla w_{j}\left(
t\right) \right\Vert _{2}^{\frac{3}{2}}\left[ \nu \lambda _{1}^{\frac{1}{4}%
}-c^{2}\underset{i=1}{\overset{n}{\sum }}\left\Vert D_{i}v_{j}\left(
t\right) \right\Vert _{2}\right] \left\Vert w_{j}\left( t\right) \right\Vert
_{2}^{\frac{1}{2}}.
\end{equation*}%
Whence follows, that if $\nu \lambda _{1}^{\frac{1}{4}}\geq c^{2}\underset{%
i=1}{\overset{3}{\sum }}\left\Vert D_{i}v_{j}\left( t\right) \right\Vert
_{2} $ then problem (1.1$^{1}$)-(\ref{3}) has only unique solution (and
solution is stable), where $\lambda _{1}$ is minimum of the spectrum of the
operator Laplace. Thus is proved

\begin{theorem}
\label{Th_3.1}Let $\Omega \in R^{3}$ be a open bounded domain of Lipschitz
class, $\left( u_{0},f\right) \in H\left( \Omega \right) \times L^{2}\left(
0,T;V^{\ast }\left( \Omega \right) \right) $ and weak solution $u\left(
t,x\right) $ of problem (1.1$^{1}$)-(\ref{3}) exists and $u\in \mathcal{V}%
\left( Q^{T}\right) $. Then if either $\underset{\Omega }{\int }\left\vert
B\left( w,w\right) \right\vert dx\geq 0$ or $\underset{\Omega }{\int }%
\left\vert B\left( w,w\right) \right\vert dx\neq 0$ (is undetermined) and $%
\nu \lambda _{1}^{\frac{1}{4}}\geq c^{2}\underset{i=1}{\overset{3}{\sum }}%
\left\Vert D_{i}u_{j}\left( t\right) \right\Vert _{2}$ fulfilled then weak
solution $u\left( t,x\right) $ is unique.\bigskip
\end{theorem}

\end{document}